\documentclass{article} 

\usepackage[english]{babel}
\usepackage{amsmath,amsfonts,amsthm}
\usepackage{wasysym}
\usepackage[dvips]{graphicx}
\usepackage{enumerate}
\usepackage{mathtools}
\usepackage{mathrsfs}
\usepackage{bbm} 
\usepackage{xcolor}
\usepackage[pdfpagemode=UseNone,bookmarksopen=false,colorlinks=true,urlcolor=blue,citecolor=blue,citebordercolor=blue,linkcolor=blue]{hyperref}
\usepackage{enumitem}

\usepackage[margin=1in]{geometry}

\usepackage[numbers]{natbib}

\usepackage{pifont}

\newcommand{\eps}{\varepsilon}
\newcommand{\al}{\alpha}
\newcommand{\Nat}{\mathbb{N}}
\newcommand{\Real}{\mathbb{R}}

\newcommand{\ex}[1]{\mathbb{E}\left[#1\right]}

\newcommand{\F}{{\mathcal F}}
\newcommand{\M}{{\mathcal M}}
\newcommand{\C}{{\mathcal C}}

\newcommand{\eul}{\mathrm{e}}
\newcommand{\Var}[1]{\mathrm{Var}\left(#1\right)}

\usepackage[margin=1in]{geometry}

\newtheorem{firsttheorem}{Proposition}

\newtheorem{theorem}[firsttheorem]{Theorem}
\newtheorem{lemma}[firsttheorem]{Lemma}
\newtheorem{corollary}[firsttheorem]{Corollary}

\newtheorem{example}[firsttheorem]{Example}

\numberwithin{equation}{section}
\numberwithin{firsttheorem}{section}

\numberwithin{secondtheorem}{section}

\newcounter{parentlemma}

\newcommand{\pr}[1]{\mathrm{Pr}\left[#1\right]}

\newcommand{\bigO}[1]{\mathcal{O}\left(#1\right)}

\newcommand{\iu}{\mathrm{i}\mkern1mu}

\newcommand{\e}[1]{\exp\left\{#1\right\}}

\DeclarePairedDelimiter\floor{\lfloor}{\rfloor}

\begin{document}
\title{Cluster Statistics in Expansive Combinatorial Structures} 
\author{Konstantinos Panagiotou\thanks{Department of Mathematics, Ludwigs-Maximilians-Universit\"at M\"unchen. E-mail: kpanagio@math.lmu.de. Funded by the European Research Council, ERC Grant Agreement 772606-PTRCSP}\, and Leon Ramzews\thanks{Department of Mathematics, Ludwigs-Maximilians-Universit\"at M\"unchen. E-mail: ramzews@math.lmu.de. Funded by the Deutsche Forschungsgemeinschaft (DFG, German Research Foundation), Project PA 2080/3-1.}}
\date{\today}
\maketitle

\begin{abstract}
We develop a simple and unified approach to investigate several aspects of the cluster statistics of random expansive (multi-)sets.
In particular, we  determine the limiting distribution of the size of the smallest and largest clusters, we establish all moments of the distribution of the number of clusters, and we prove a local limit theorem for that distribution.
Our proofs combine effectively two simple ingredients: an application of the saddle-point method through the well-known framework of $H$-admissibility, and an ingenious idea by Erd\H{o}s and Lehner that utilizes the elementary inclusion/exclusion principle.


\end{abstract}

\tableofcontents
\newpage
\section{Introduction \& Main Results}

\paragraph{Sets and Multisets}


In a particular setting of the models that we consider we are given a set $\C$ -- the so-called \emph{clusters} or \emph{components} -- whose elements have each a specific size, that is, we are given in addition a mapping $|\cdot|:\C\to\Nat$ that assigns to each cluster its size.
Then we study objects that are composed out of elements in $\C$ and whose structure can be described in terms of elements in the set
\begin{align*}
    \Omega_n :=
    \Big\{
        (N_1,\dots,N_n)\in\Nat_0^n:
        \sum_{1\le k\le n}kN_k = n
    \Big\}.
\end{align*}
Using this notation, we say that a (compound) object $O$ has cluster structure $(N_1,\dots,N_n) \in \Omega_n$, if there are $N_k$ clusters of size $k$ in $O$ for $1\le k\le n$. Naturally, in that case we say that $O$ has (total) size $n$ and we write $|O| = n$.

The first broad class of models that admit such a description are \emph{sets} or \emph{assemblies} on the label set
$[n] := \{1, 2, \dots , n\}$. This class includes the prominent examples of permutations (where the clusters are cycles),  graphs that have some appropriate property (where the clusters are the connected graphs with that property) and the  Maxwell-Boltzmann model
of ideal gas from statistical physics (where the clusters are particles at some energy level).
Let us write $C_k$ for the number of objects of size $k$ in $\C$, $k\in \Nat$.
Then, in this class of models the number of compound objects (that is, sets) with cluster structure $(N_1, \dots,N_n) \in \Omega_n$ equals
\begin{equation}
\label{eq:setswithclusterstrucure}
    \frac{n!}{\prod_{1\le i \le n}i!^{N_i}} \cdot
    \prod_{1\le i\le n} \frac{C_i^{N_i}}{N_i!}
    = n! \prod_{1\le i\le n} \frac{(C_i/i!)^{N_i}}{N_i!};
\end{equation}
the first term is a multinomial coefficient that partitions the $n$ labels to the $N_i$ objects of size $i$, and the second term selects
a set of $N_i$ objects from the objects of size $i$ in $\C$, for all $1\le i \le n$. Actually, this formula can also be used if we replace $C_i$ with an arbitrary non-negative real number (with no combinatorial interpretation); this is the case, for example, in coagulation-fragmentation processes, where~\eqref{eq:setswithclusterstrucure} is related to the equilibrium state of the system. 
We refer to \cite{Granovsky2013} for an excellent overview and many detailed examples, and to the extensive work~\cite{Joyal1981} and the excellent books~\cite{Flajolet2009,Bergeron1998} that treat the combinatorial settings. 

From now on we assume that the $C_k$'s are non-negative real numbers.
If we sum up~\eqref{eq:setswithclusterstrucure} over all elements in~$\Omega_n$, we obtain the \emph{partition function} of the model, which in the combinatorial setting is just the number of $\C$-sets of total size $n$. Then, a fundamental fact is, see~\cite{Flajolet2009,Bergeron1998}, that if we set $c_k = C_k/k!$, the partition function equals $n!$ times the $n$-th coefficient in a power series $S$ that satisfies the beautiful and simple identity
\begin{align}
    \label{eq:ogf_set}
    C(x) := \sum_{k\ge 1} c_k x^k
    = \sum_{k\ge 1} \frac{C_k}{k!} x^k
    \quad
    \text{and}
    \quad
    S(x) := e^{C(x)}.
\end{align}
A further important class of models are \emph{multisets}, where we may add each cluster several times. This class includes -- most prominently -- partitions of an integer (where the clusters are the natural numbers) and unlabeled objects, for example graphs (where the clusters are unlabeled connected graphs) as well as the Bose-Einstein model of ideal gas. In this class of models the number of compound objects (that is, multisets) with cluster structure $(N_1, \dots,N_n) \in \Omega_n$ equals
\begin{equation}
\label{eq:multisetswithclusterstrucure}
    \prod_{1\le i\le n} \binom{C_i + N_i - 1}{N_i},
\end{equation}
where the binomial coefficient accounts for the number of ways to choose with repetition $N_i$ elements from~$\C_i, 1\le i \le n$. As before,~\eqref{eq:multisetswithclusterstrucure} makes also sense if the $C_k$'s are arbitrary non-negative reals, which happens for example in another model of coagulation-fragmentation processes. Again~\cite{Granovsky2006} is a great source of examples for that matter. Moreover, summing up~\eqref{eq:multisetswithclusterstrucure} over all elements in~$\Omega_n$, we obtain the partition function of that model, which in the combinatorial setting equals the number of $\C$-multisets of total size $n$. Then, if we set $c_k = C_k$, the partition function equals the $n$-th coefficient of a power series $G$ that satisfies the (again, beautiful) identity
\begin{align}
    \label{eq:ogf_mset}
    C(x) := \sum_{k\ge 1}c_k x^k,
    \quad´
    \text{and}
    \quad
    G(x) := \exp\left\{\sum_{j \ge 1} C(x^j)/j \right\}.
\end{align}
If we consider structures as the ones just described, then two questions immediately pop up. First, there is the counting problem: how do the coefficients of $C$ relate to the coefficients of $S$ resp.~$G$, in other words, how many $\C$-(multi-)sets of a given size are there? This question alone is mathematically challenging and interesting, but further it is essential to investigate a second question: what can we say about the ``typical'' cluster structure in a (multi-)set? More specifically, let $\mathsf{S}^{(n)}=(\mathsf{S}^{(n)}_1,\dots,\mathsf{S}^{(n)}_n) \in \Omega_n$ and $\mathsf{G}^{(n)} = (\mathsf{G}^{(n)}_1,\dots,\mathsf{G}^{(n)}_n)\in\Omega_n$ denote the cluster structures of an element drawn uniformly at random from all $\C$-sets and $\C$-multisets of (total) size~$n$.
Denote by $[x^n]F(x)$ the coefficient of $x^n$ in a power series $F$.
By definition and using~\eqref{eq:setswithclusterstrucure} and~\eqref{eq:multisetswithclusterstrucure} we  readily obtain
\begin{align} \label{eq:distribution_cluster_set}
	\pr{\mathsf{S}^{(n)}=(N_1,\dots,N_n)} 
	&= \frac{1}{[x^n]S(x)} \cdot \prod_{1\le i\le n}\frac{c_i^{N_i}}{N_i!} \cdot \mathbbm{1}_{(N_1,\dots,N_n)\in\Omega_n}, \\ \label{eq:distribution_cluster_multiset}
	\pr{\mathsf{G}^{(n)}=(N_1,\dots,N_n)} 
	&= \frac{1}{[x^n]G(x)} \cdot \prod_{1\le i\le n}\binom{c_i+N_i-1}{N_i}
	\cdot \mathbbm{1}_{(N_1,\dots,N_n)\in\Omega_n}.
\end{align}
Let $\mathsf{F}^{(n)} = (\mathsf{F}^{(n)}_1,\dots,\mathsf{F}^{(n)}_n)$ be either $\mathsf{S}^{(n)}$ or $\mathsf{G}^{(n)}$.
Central objects of interest are the distribution of the \emph{total} number of clusters
\[
    \kappa(\mathsf{F}^{(n)}) 
    := \sum_{1\le k\le n} \mathsf{F}_k^{(n)}
\]
and the distribution of the size of the smallest and largest clusters
\begin{align*}
	\M(\mathsf{F}^{(n)}) := \min \left\{1\le k \le n:  \mathsf{F}^{(n)}_k>0 \right\}
	\quad\text{and}\quad 
	\mathcal{L}(\mathsf{F}^{(n)}) := \max \left\{1\le k \le n:  \mathsf{F}^{(n)}_k>0\right\}.
\end{align*}

\paragraph{State of the art} Let $h:[1,\infty) \to [0,\infty)$ be an eventually positive, continuous and slowly varying function, which means that
\[
    \lim_{x\to\infty} \frac{h(\lambda x)}{h(x)} = 1, \quad \text{for }\lambda>0.    
\]
Then a quite common and rather general assumption in the study of (multi-)sets is that the counting sequence for $\C$ fulfils
\begin{equation}
\label{eq:cn}
	c_n = h(n)\cdot n^{\al-1}\cdot \rho^{-n},\quad \al\in\Real, ~ 0<\rho\le1.
\end{equation}
Under this assumption the prominent problem of determining $[x^n]S(x)$ and $[x^n]G(x)$ is treated in several papers. The state of the art is:~\cite{Stufler2020,Stufler2018} treat the \emph{convergent} case $\al<0$,~\cite{Arratia2003} the \emph{logarithmic} case $\al=0$, \cite{Granovsky2006,Freiman2002} the \emph{expansive} case $\al>0$ and~\cite{Granovsky2008} a specific ``quasi-expansive'' class of sequences that satisfy an analytical set of conditions implying $\al>0$ and $\rho=1$, the \emph{Meinardus scheme of conditions}. We see that from today's viewpoint the partition functions are pretty well understood. 

Regarding the statistical properties of (multi-)sets the picture in the convergent case ($\alpha < 0$) is also rather complete. It is known that the number $\kappa(\mathsf{F}^{(n)})$ of components converges (without scaling) in distribution, see~\cite{Bell2000}, and the tails are determined completely in~\cite{Panagiotou2020}. In this case, removing a -- the --  largest cluster from $\mathsf{F}^{(n)}$ leaves an object that converges in distribution  to a limit given by the Boltzmann model, cf.~\cite{Stufler2020,Stufler2018,Barbour2005}, implying that $n - \mathcal{L}(\mathsf{F}^{(n)})$ has a limiting distribution.

The expansive case, on the other side, is quite different from that. In general, $\kappa(\mathsf{F}^{(n)})$ is unbounded with high probability in this case. 
This is already reflected in the special case of integer partitions, which are multisets of natural numbers;  the classical work by~\citet{Erdoes1941} determines the proper normalization so that the number of parts in a random integer partition converges in distribution.
However, in the general expansive setting, the precise picture is not completely understood and rather fragmented. 
More specifically, the number of $\C$-multisets of size $n$ and $N$ clusters, which is obviously related to $\Pr[\kappa(\mathsf{G}^{(n)}) = N]$, was studied in~\cite{Panagiotou2022} when $\rho < 1$.
The distribution of $\kappa(\mathsf{G}^{(n)})$ was studied in the quasi-expansive case:~\cite{Stark2021} describes parts of the distribution, and~\cite{Mutafchiev2011} finds the proper normalization leading to a central limit theorem, which is Gaussian only for parts of the parameter range. Determining the proper scaling for $\kappa(\mathsf{G}^{(n)})$ in the expansive case is still an open problem.
Regarding sets, the proper normalization that yields a local and central Gaussian limit theorem for $\kappa(\mathsf{S}^{(n)})$ is established in~\cite{Erlihson2004} in the expansive case and when $h(n)$ is constant; treating general $h$ has not been achieved so far.


There are also many results about the smallest and largest clusters of $\mathsf{F}^{(n)}$. Here integer partitions are also well-understood, as the representation of integer partitions into Ferres diagrams shows that the distribution of the number of clusters and of the largest clusters are identical.
In the case of multisets, \citet{Mutafchiev2013} establishes that after appropriate normalization $\mathcal{L}(\mathsf{G}^{(n)})$ converges to an extreme value distribution in the quasi-expansive case. Regarding the expansive case we are not aware of any detailed results concerning $\mathcal{M}(\mathsf{G}^{(n)})$ and $\mathcal{L}(\mathsf{G}^{(n)})$.
For the case of $\C$-sets,~\cite{Freiman2005} proved that there is a threshold $n^{1/(\al+1)}$ for $\mathcal{L}(\mathsf{S}^{(n)})$, meaning that the probability of the event $\{\mathcal{L}(\mathsf{S}^{(n)}) \le n^{\beta}\}$ tends to $0$ if $\beta<1/(\al+1)$ and to $1$ if $\beta>1/(\al+1)$. In their work also the limiting distribution of the minimal cluster size is determined.  still unknown.

\paragraph{Our contribution} In this paper we perform a thorough study of the expansive case, that is, we make an assumption that is at least as general as~\eqref{eq:cn} with $\alpha > 0$ -- see below -- and we derive a rather complete picture: we determine the limiting distribution of the size of the smallest and largest clusters, we establish all moments of the distribution of the number of clusters, and we prove a local limit theorem for that distribution for both $\C$-sets and $\C$-multisets. 

On the methodological side, our findings are based on a simple, yet far-reaching observation, that somehow has escaped notice so far. The proofs in the state of the art works, for example~\cite{Freiman2002,Granovsky2006}, are conducted by studying the partition functions $[x^n]S(x)$ and $[x^n]G(x)$ with the help of the powerful \emph{Khinchin's probabilistic method},  see~\cite{Freiman2005} for a great overview and some historical perspective, to reduce the analytical problem to a probabilistic one; but \emph{in heart} the authors apply the well-known saddle-point method to the Cauchy integral representing the coefficients of the series at hand.
For that they basically prove that $S(x)$ and $G(x)$ possess a property called \emph{$H$-admissibility}, see Section~\ref{sec:H-admissibility}, but without calling it that way. Now, the great advantage in realising that a series $F$ is $H$-admissible is that by existing results one is able to compute/estimate $[x^n]F(x)$ and $[x^{n-k}]F(x)/[x^n]F(x)$ systematically for virtually \emph{any} $0<k\le n$ as $n\to\infty$.
This has powerful consequences, as the studies of the cluster statistics always rely at some point on determining such a fraction. So, knowing that~$S$,~$G$ -- and for that effect, many related series that we construct explicitly -- are $H$-admissible opens us the path for obtaining incredibly detailed results. In particular, by combining our observation with the elementary methods developed in~\cite{Erdoes1941} in the context of integer partitions, we are able to find the proper scaling under which  $\mathcal{L}(\mathsf{S}^{(n)})$, $\mathcal{L}(\mathsf{G}^{(n)})$ and  $\mathcal{M}(\mathsf{S}^{(n)})$, $\mathcal{M}(\mathsf{G}^{(n)})$, 
converge and we identify the limit.

\paragraph{Plan of the paper}
In the remainder of this section we present our main results. In Section~\ref{sec:H-admissibility} we recall from the existing literature the concept of $H$-admissibility, where we also present the main tool,  Lemma~\ref{lem:hayman}, that provides  coefficient extraction results for $H$-admissible functions. Section~\ref{sec:proofs} contains all  proofs. There we establish in Section~\ref{sec:proof_h_admissibility} the result that $S,G$ and some  related and directly relevant power series are indeed $H$-admissible. Subsequently, the main results are proven in Section~\ref{sec:proof_main}.

\subsection{Main Results}
\label{sec:main_results}

This section is structured as follows. We begin with introducing \emph{oscillating expansive} sequences, that constitute a setting more general than~\eqref{eq:cn} when $\alpha > 0$. In our first result Theorem~\ref{thm:coeff_eC(x)_C(x)^k}, which is of \emph{technical} nature, we show that we can extract the coefficients of several generating series related to $S(x)$ and $G(x)$. 
As a first application we determine in Theorem~\ref{thm:largest_comp_hayman} and Corollary~\ref{coro:smallest_comp_hayman} the distribution of the largest and smallest components in (multi-)sets drawn uniformly at random.
Moreover, as a rather straightforward consequence of Theorem~\ref{thm:coeff_eC(x)_C(x)^k}, we determine the moments of the number of components in uniform \hbox{(multi-)}sets in Corollary~\ref{coro:mean_sec_moment_hayman}. 
Finally, we establish a local limit law for the number of components in Theorem~\ref{thm:LLT_number_comp_hayman}, a fact whose proof utilizes the aforementioned enumeration results~\cite{Panagiotou2022} and the acquainted knowledge about the moments of the distribution. 

We will use the following notation. A non-negative sequence $(c_k)_{k\in\Nat}$ is said to be in the set $\F(\al_1,\al_2,\rho)$ for $0<\al_1<\al_2$ and $\rho>0$ if there are constants $0<A_1<A_2$ such that for all sufficiently large $n$ 
\begin{align}
    \label{eq:oscilatting_expansive}
	A_1\cdot n^{\al_1-1}\cdot \rho^{-n} \le c_n \le A_2\cdot n^{\al_2-1} \cdot \rho^{-n}.
\end{align}
Such sequences are referred to as \emph{oscillating expansive} with parameters $\al_1,\al_2,\rho$. 
This setting is rather broad, and it also includes models that are not expansive in the sense that
\[
	c_n
	= h(n)\cdot n^{\al-1}\cdot \rho^{-n},
	\quad h\text{ is eventually positive and slowly varying},\,\al>0,\,0<\rho\le1,\,n\in\Nat,
\]
for example ideal gases; we refer for a general introduction and examples to see~\cite{Granovsky2013}.

\paragraph{Remark: Real-valued sequences.} As we already remarked, our upcoming results all hold for arbitrary non-negative real-valued sequences $(c_k)_{k\in\Nat}$ that do not necessarily have a combinatorial interpretation. This is possible since $S(x)$ and $G(x)$ can be directly defined by~\eqref{eq:ogf_set} and~\eqref{eq:ogf_mset} without being seen as the generating series of sets and multisets. Moreover, by viewing the measures in~\eqref{eq:distribution_cluster_set} and~\eqref{eq:distribution_cluster_multiset} as the \emph{definition} of $\mathsf{S}^{(n)}$ and $\mathsf{G}^{(n)}$, these random variables can be investigated detached from uniform sets and multisets.
Nevertheless, in the rest of the paper, whenever we talk about the ``set'' setting, we always refer to $S(x)$ and $\mathsf{S}^{(n)}$ and similarly, when we talk about ``multisets'', we will refer to $G(x)$ and $\mathsf{G}^{(n)}$.

\paragraph{Coefficient Extraction and Counting}

In our first main result we determine asymptotically the coefficients of $S,G$ and a wide class of related power series in the oscillating expansive setting. 
\begin{theorem}
\label{thm:coeff_eC(x)_C(x)^k}
For $\al>0$, $0<\rho\le 1$ and $0<\eps<\al/3$ let $(c_k)_{k\in\Nat}\in\F(2\al/3+\eps,\al,\rho)$. Let $z_n$ be the solution to $z_nC'(z_n)=n$. Then, for $\ell\in\Nat_0$ and as $n\to\infty$,
\begin{align}
    \label{eq:coeff_egf_set}
    [x^n]S(x)\cdot C(x)^\ell
    \sim \frac{S(z_n)C(z_n)^\ell}{\sqrt{2\pi C''(z_n)}} \cdot z_n^{-n-1}.
\end{align}
Let $\ell\in\Nat_0, (p_1,\dots,p_\ell)\in\Nat_0^\ell$. If $0<\rho<1$, then as $n\to\infty$,
\begin{align}
    \label{eq:coeff_ogf_muliset_rho<1}
    [x^n]G(x) \prod_{1\le i\le \ell}\sum\nolimits_{j\ge 1}j^{p_i}C(x^j)
    \sim \e{\sum\nolimits_{j\ge 2}C(\rho^j)/j} 
    \cdot  \frac{S(z_n)C(z_n)^\ell}{\sqrt{2\pi C''(z_n)}} \cdot z_n^{-n-1}.
\end{align}
If $\rho=1$, let $q_n$ be the solution to $\sum_{j\ge 1} q_n^jC'(q_n^j) = n$. Then, as $n\to\infty$,
\begin{align}
    \label{eq:coeff_ogf_muliset_rho=1}
    [x^n]G(x) \prod_{1\le i\le \ell}\sum\nolimits_{j\ge 1}j^{p_i}C(x^j)
    \sim  \frac{G(q_n) \prod_{1\le i\le \ell}\sum\nolimits_{j\ge 1}j^{p_i}C(q_n^j)}{\sqrt{2\pi \sum\nolimits_{j\ge 1}jq_n^{2j}C''(q_n^j)}} \cdot q_n^{-n}.
\end{align}
\end{theorem}
The proof is in Section~\ref{sec:coro:coeff_eC(x)_C(x)^k}.
A natural application of Theorem~\ref{thm:coeff_eC(x)_C(x)^k} is solving the counting problem, that is, determining the numbers $s_n/n!:=[x^n]S(x)$ and $g_n:=[x^n]G(x)$ of $\C$-sets and $\C$-multisets of size $n$ asymptotically. That is, if $(c_k)_{k\in\Nat}\in\F(2\al/3+\eps,\al,\rho)$ for $\al>0,0<\eps<\al/3$ and $0<\rho\le 1$ we recover (though, due to the notational discrepancy this is not immediately obvious) streamlined versions of the results in~\cite{Freiman2002,Granovsky2006} 
\[
    s_n
    \sim \frac{S(z_n)}{\sqrt{2\pi z_n^2C''(z_n)}} \cdot z_n^{-n} \cdot n!, \quad z_n\text{ solves }z_nC'(z_n)=n
\]
and
\begin{align*}
    g_n
    \sim \begin{dcases*}
        \exp\bigg\{\sum_{j\ge 2}C(\rho^j)/j\bigg\} \frac{\e{C(z_n)}}{\sqrt{2\pi z_n^2C''(z_n)}} \cdot z_n^{-n}, &$0<\rho<1$ and $z_n$ solves $z_nC'(z_n)=n$ \\
        \frac{G(q_n)}{\sqrt{2\pi \sum_{j\ge 1}jq_n^{2j}C''(q_n^j)}}\cdot q_n^{-n} , &$\rho = 1$ and $q_n$ solves $\sum_{j\ge 1}q_n^jC'(q_n^j) = n$
    \end{dcases*}.
\end{align*}

\paragraph{The Distribution of the Largest and Smallest Clusters}
Theorem~\ref{thm:coeff_eC(x)_C(x)^k} is not only helpful to obtain counting results, but it can also be applied to get fine grained information about the cluster statistics of random (multi-)sets. Let $\mathsf{S}^{(n)}$ and $\mathsf{G}^{(n)}$ be the random cluster structures from~\eqref{eq:distribution_cluster_set} and~\eqref{eq:distribution_cluster_multiset}, respectively.
The next statements are concerned with the distribution of the extreme (in both directions of the spectrum) cluster sizes in $\mathsf{S}^{(n)}$ and $\mathsf{G}^{(n)}$. Recall that for $\mathsf{F}^{(n)}\in\{\mathsf{S}^{(n)},\mathsf{G}^{(n)}\}$
\[
	\M(\mathsf{F}^{(n)}) 
	:= \min\left\{1\le k\le n:\mathsf{F}^{(n)}_k>0\right\}
	\quad \text{and}\quad 
	\mathcal{L}(\mathsf{F}^{(n)})
	:= \max\left\{1\le k\le n:\mathsf{F}^{(n)}_k>0\right\}.
\]	
First we treat the largest clusters. \citet{Freiman2005} show that if $(c_k)_{k\in\Nat}$ is expansive, then the coarse threshold for the size of the largest cluster in the $\C$-set $\mathsf{S}^{(n)}$ is given by $n^{1/(\al+1)}$, in the sense that
\begin{align*}
    \lim_{n\to\infty}\pr{\mathcal{L}(\mathsf{S}^{(n)}) \le n^\beta}
    = \begin{dcases*}
    0, &$\beta<1/(\al+1)$ \\
    1, &$\beta>1/(\al+1)$
    \end{dcases*}.
\end{align*}
However, the question about the actual order of magnitude, and even more, the limiting distribution, remained an open problem. \citet{Mutafchiev2013} treated the $\C$-multiset $\mathsf{G}^{(n)}$ in a specific setting -- under the aforementioned Meinardus scheme of conditions -- that in addition to several assumptions requires $\rho=1$. He established that for specific functions functions $f(n),g(n)$
\[
	\lim_{n\to\infty} \pr{\mathcal{L}(\mathsf{G}^{(n)}) \le g(n) + t f(n)}
	= \eul^{-\eul^{-t}}, \quad t\in\Real,
\]
a fact that was known for integer partitions from the classical work~\cite{Erdoes1941}. In our next result we show that, after appropriate normalization, this behavior is universal for expansive (multi-)sets.
\begin{theorem}
\label{thm:largest_comp_hayman}
Let $(c_k)_{k\in\Nat}$ be expansive. For $t\in\Real$ and $\beta>0$ set
\[
    s(t,\beta)
    := \beta^{-1} \big(\ln X + t\big),
    \quad \text{where} \quad 
    X = \Gamma(\al)^{-1} C(\rho\eul^{-\beta}) (\ln C(\rho\eul^{-\beta}))^{\al-1} \frac{h(\beta^{-1}\ln C(\rho\eul^{-\beta}))}{h(\beta^{-1})}.
\]
Let $z_n=\rho\eul^{-\eta_n}$ be the solution to $z_nC'(z_n) = n$ and $q_n=\rho\eul^{-\xi_n}$ the solution to $\sum_{j\ge 1}q_n^jC'(q_n^j)=n$.
Then
\begin{align*}
    \lim_{n\to\infty}\pr{\mathcal{L}(\mathsf{F}^{(n)})\le s(t,\beta_n)} 
    = \eul^{-\eul^{-t}},
     \text{ where }\beta_n =\begin{dcases*}
    \eta_n, &$\mathsf{F}^{(n)}=\mathsf{S}^{(n)}\text{ or }\mathsf{F}^{(n)}=\mathsf{G}^{(n)},0<\rho<1$ \\
    \xi_n, &$\mathsf{F}^{(n)}=\mathsf{G}^{(n)},\rho=1$
    \end{dcases*},
    \quad t\in\Real.
\end{align*}
\end{theorem}
The proof, which is in Section~\ref{sec:thm:largest_comp_hayman}, is based on a combination of $H$-admissibility together with the inclusion/exclusion principle as exploited in~\cite{Erdoes1941}.
For an application of Theorem~\ref{thm:largest_comp_hayman} we have to determine $C(\rho\eul^{-\beta})$ and $\rho\eul^{-\beta} C'(\rho\eul^{-\beta})$ accurately enough as $\beta\to 0$. Here is a particular example where this works out nicely, the proof of which is performed after the proof of Theorem~\ref{thm:largest_comp_hayman}.
\begin{example}
\label{example:largest_comp} 
Let $c_n=n^{\al-1}\rho^{-n}$ for $\al>0,0<\rho\le 1$ and $n\in\Nat$. Define 
\[
    f(n) := \left(\frac{n}{\Gamma(\al+1)}\right)^{1/(\al+1)}
    \quad\text{and}\quad \tilde{f}(n)
    :=\left(\frac{n}{\Gamma(\al+1)\zeta(\al+1)}\right)^{1/(\al+1)}.
\]
Then, if $z_n=\rho\eul^{-\eta_n}$ is the solution to $z_nC'(z_n)=n$,
\begin{align*}
    \eta_n=f(n)^{-1} + o(n^{-1}) \quad\text{and}\quad 
    \ln X 
    = \al\ln f(n) + (\al-1)\ln\ln f(n) + (\al-1)\ln \al+ o(1)
\end{align*}
and, if $q_n=\rho\eul^{-\xi_n}$ is the solution to $\sum_{j\ge 1}q_n^jC'(q_n^j)=n$,
\[
    \xi_n = \tilde{f}(n)^{-1} + o(n^{-1}) \quad\text{and}\quad 
    \ln X = \al\ln\tilde{f}(n) + (\al-1) \ln\ln \tilde{f}(n) + (\al-1)\ln\al + o(1).
\]
This delivers the exact scaling for the distribution of the largest cluster in $\mathsf{S}^{(n)}$ and $\mathsf{G}^{(n)}$ for all $0<\rho\le 1$ .
\end{example}
Next we consider the smallest clusters. In the next result we determine the distribution of $\mathcal{M}(\mathsf{S}^{(n)})$, $\mathcal{M}(\mathsf{G}^{(n)})$. The authors of~\cite{Freiman2005} determine the limiting distribution of $\mathcal{M}(\mathsf{S}^{(n)})$ in the expansive case. We extend their result to the oscillating expansive case and further to $\mathcal{M}(\mathsf{G}^{(n)})$.
\begin{corollary}
\label{coro:smallest_comp_hayman}
For $\al>0$, $0<\rho\le 1$ and $0<\eps<\al/3$ let $(c_k)_{k\in\Nat}\in\F(2\al/3+\eps,\al,\rho)$. Then, for $s\in\Nat$,
	\[
		\lim_{n\to\infty}\pr{\mathcal{M}(\mathsf{F}^{(n)}) > s}
		= 
		\begin{dcases*}
		\exp\bigg\{-\sum_{1\le k\le s}c_k\rho^k\bigg\}, &$\mathsf{F}^{(n)} = \mathsf{S}^{(n)}$ \\
		\exp\bigg\{-\sum_{j\ge 1}\sum_{1\le k\le s}c_k\rho^{jk}/j\bigg\}, &$\mathsf{F}^{(n)}=\mathsf{G}^{(n)}$
		\end{dcases*}.
	\]
	Moreover, $\lim_{n\to\infty}\pr{\mathcal{M}(\mathsf{F}^{(n)})> s_n}=0$ for any $s_n\to\infty$ and $\mathsf{F}^{(n)}\in\{\mathsf{S}^{(n)},\mathsf{G}^{(n)}\}$.
\end{corollary}

\paragraph{The Cluster Distribution}

In this section we study the distribution of the number of clusters in $\mathsf{S}^{(n)}$ and $\mathsf{G}^{(n)}$. Recall that for $\mathsf{F}^{(n)}\in\{\mathsf{S}^{(n)},\mathsf{G}^{(n)}\}$ the number of components is defined by
\[
	\kappa(\mathsf{F}^{(n)}) 
	:= \sum_{1\le k\le n}\mathsf{F}^{(n)}_k.
\]
Extending the generating series $S(x)$ and $G(x)$ to the bivariate version where the additional variable $y$ tags the number of clusters in a (multi-)set yields, see~\cite{Flajolet2009,Bergeron1998},
\[
    S(x,y) 
    =\e{yC(x)}
     \quad\text{and}\quad 
    G(x,y)
    =\e{\sum_{j\ge 1}\frac{C(x^j)y^j}{j}}.
\]
Then for $k\in\Nat$
\begin{align*}
	\pr{\kappa(\mathsf{S}^{(n)})=k} = \frac{[x^ny^k]S(x,y)}{[x^n]S(x)}
	\quad\text{and}\quad
	\pr{\kappa(\mathsf{G}^{(n)})=k} = \frac{[x^ny^k]G(x,y)}{[x^n]G(x)}.
\end{align*}
For $n,\ell\in\Nat$ let $(n)_\ell$ denote the falling factorial $(n)_\ell := n (n-1) \cdots (n-\ell+1)$.
By a straightforward computation we obtain the well-known relations
\begin{align}
	\label{eq:number_comp_falling_factorial_hayman}
    \ex{(\kappa(\mathsf{S}^{(n)}))_\ell} 
    = \frac{[x^n]d^\ell/(dy^\ell)S(x,y)\big|_{y=1}}{[x^n]S(x)} 
    \quad \text{and}\quad 
    \ex{(\kappa(\mathsf{G}^{(n)}))_\ell} 
    = \frac{[x^n]d^\ell/(dy^\ell)G(x,y)\big|_{y=1}}{[x^n]G(x)},
    \quad \ell\in\Nat.     
\end{align}
For the specific case $c_n\sim cn^{\al-1}\rho^{-n}$ and $c,\al>0,0<\rho\le 1$ the results in~\citet{Erlihson2008} can be used to obtain the moments of $\kappa(\mathsf{S}^{(n)})$; the following result completes the picture for both sets and multisets in the oscillating expansive case.
\begin{corollary}
	\label{coro:mean_sec_moment_hayman}
    For $\al>0$, $0<\rho\le 1$ and $0<\eps<\al/3$ let $(c_k)_{k\in\Nat}\in\F(2\al/3+\eps,\al,\rho)$. Let $z_n$ be the solution to $z_nC'(z_n)=n$. Then
    \begin{align*}
    	\ex{\kappa(\mathsf{F}^{(n)})^\ell}
    	\sim C(z_n)^\ell
    	\quad\text{for}\quad
    	\mathsf{F}^{(n)}
    	=
    	\begin{dcases*} 
    	\mathsf{S}^{(n)}, &$0<\rho\le 1$ \\
    	\mathsf{G}^{(n)}, &$0<\rho<1$
    	\end{dcases*}.
	\end{align*}
    If $\rho=1$, let $q_n$ be the solution to $\sum_{j\ge 1}q_n^jC'(q_n^j)=n$. Then
    \[
    	\ex{\kappa(\mathsf{G}^{(n)})}
    	\sim \sum_{j\ge 1}C(q_n^j) \quad\text{and}\quad 
    	\ex{\kappa(\mathsf{G}^{(n)})^2}
    	\sim \Big(\sum_{j\ge 1}C(q_n^j)\Big)^2 + \sum_{j\ge 1}jC(q_n^j).
    \]
\end{corollary} 
The proof is in Section~\ref{sec:coro:mean_sec_moment_hayman}. Some further comments are in place.
First, note that in Corollary~\ref{coro:mean_sec_moment_hayman} we consider only the first and second moment of $\kappa(\mathsf{G}^{(n)})$ in the case $\rho = 1$. We actually may compute \emph{any} moment also in that case by determining carefully the derivatives in~\eqref{eq:number_comp_falling_factorial_hayman} and then making use of Theorem~\ref{thm:coeff_eC(x)_C(x)^k}.
However, already for the case $\ell=2$ the second term in the asymptotic formula for $\ex{\kappa(\mathsf{G}^{(n)})^2}$ can be of the same order as the first term. To see this, consider following example. Let $c_n \sim n^{\al-1}$ for some $0<\al<1$. Set $q_n=\eul^{-\xi_n}$. Then, as we will show later in Lemma~\ref{lem:sum_k^a/(1-x)^b_different_regimes_proof_hayman}, we have that $\sum_{j\ge 1}C(q_n^j) = \Theta(\xi_n^{-1})$ and $\sum_{j\ge 1}jC(q_n^j)=\Theta(\xi_n^{-2})$. For this reason it seems out of reach to get a ``nice'' formula for $\ex{\kappa(\mathsf{G}^{(n)})^\ell}$ for general $\ell$. So, we are content with only presenting the cases $\ell=1,2$.
As a second remark, we want to mention that in general we cannot compute the variance in any of the cases that we consider, as we do not know (and in the generality considered here, cannot obtain) the second asymptotic order of the expressions at hand.

Under the stronger assumption that $(c_k)_{k\in\Nat}$ is expansive we can actually say much more. In this direction, \citet{Erlihson2004} derived a local and a central limit theorem for $\kappa(\mathsf{S}^{(n)})$ under the condition $c_n=cn^{\al-1}\rho^{-n}$ for $c,\al>0,0<\rho\le1$ and $n\in\Nat$; note the ``='' sign. Here we establish the validity of a local limit theorem in the expansive case for sets and multisets. 
\begin{theorem}
\label{thm:LLT_number_comp_hayman}
Let $(c_k)_{k\in\Nat}$ be expansive. Then for any $K>0$, as $n\to\infty$
\begin{align*}
	\pr{\kappa(\mathsf{S}^{(n)}) = \floor{C(z_n) + t\sqrt{C(z_n)/(\al+1)}}} 
	\sim \eul^{-t^2/2}\cdot \frac{1}{\sqrt{2\pi C(z_n)/(\al+1)}}, \quad t\in[-K,K].
\end{align*}
If $0<\rho<1$ we further get that for any $t\in\Real$, as $n\to\infty$,
\begin{align*}
	\pr{\kappa(\mathsf{G}^{(n)}) = \floor{C(z_n) + t\sqrt{C(z_n)/(\al+1)}}} 
	\sim \eul^{-t^2/2}\cdot \frac{1}{\sqrt{2\pi C(z_n)/(\al+1)}}.
\end{align*}
\end{theorem}
The proof can be found in Section~\ref{sec:thm:bivariate_counting_set_hayman}. Let us make two more remarks. First, this theorem of course \emph{strongly suggests} that the variance of $\kappa(\mathsf{S}^{(n)})$ and $\kappa(\mathsf{G}^{(n)})$ is  $\sim C(z_n)/(\al+1)$; we leave it as an open problem to establish that. Moreover, note that in Theorem~\ref{thm:LLT_number_comp_hayman} the statement about $\kappa(\mathsf{S}^{(n)})$ implies a central limit theorem, whereas the statement about $\kappa(\mathsf{G}^{(n)})$ is weaker and in general not sufficient to obtain a central limit theorem.

In the proof we will, among other ingredients, exploit knowledge about the (asymptotic) number of~$\C$-(multi-)sets of total size $n$ and with $N$ clusters. In the paper~\cite{Panagiotou2022} we determined the asymptotic number of such multisets, which is nothing else than the coefficient $[x^ny^N]G(x,y)$, for a wide range (but not all) of~$N$. Here we solve, again by using tools from~\cite{Panagiotou2022}, the analogous problem for the set case; the proof is in Section~\ref{sec:thm:bivariate_counting_set_hayman}. 
\begin{theorem}
\label{thm:bivariate_counting_set_hayman}
Let $(c_k)_{k\in\Nat}$ be expansive. Let $N=N_n$ be such that $N,n/N\to\infty$ as $n\to\infty$. Set $r_n$ to be the solution to $r_nC'(r_n)/C(r_n)=n/N$.
Then, as $n\to\infty$,
\[
    [x^ny^{N}]S(x,y)
    \sim  \frac{1}{N!}\cdot \frac{C(r_n)^{N}}{\sqrt{2\pi N r_n^2 C''(r_n)/((\al+1)C(r_n))}} \cdot r_n^{-n}.
\]
\end{theorem}



\section{Preliminaries}
\label{sec:preliminaries}
In this section we first introduce the concept of $H$-admissibility and state  Lemma~\ref{lem:hayman}, which is a central result about coefficient extraction. Subsequently, we collect various estimates for power series that are  needed in the forthcoming proofs in Section~\ref{sec:power_series_hayman}.

\subsection{$H$-admissibility}
\label{sec:H-admissibility}
We start by reviewing the concept of $H$-admissibility, which is a general set of conditions on a function $F(x)$ with radius of convergence $0<\rho\le \infty$ under which $[x^n]F(x)$ can be  computed asymptotically. This theory was initiated in the seminal paper~\cite{Hayman1956} and has seen numerous extensions and applications. As a general reference we recommend~\cite[Ch.~VIII.5]{Flajolet2009}.

Set $F(x)=\eul^{f(x)}$.
By applying Cauchy's coefficient formula and switching to polar coordinates we obtain for some $0<r<\rho$ \begin{align}
    \label{eq:cauchy_integral_hayman_def}
    [x^n]F(x)
    = \frac{r^{-n}}{2\pi}\int_{-\pi}^\pi 
	F(r\eul^{\iu\theta}) \eul^{-n\iu\theta}d\theta.
\end{align}
To get a grip on this expression we expand $F(r\eul^{\iu\theta})$ at $\theta=0$, so that for some $\lvert\xi\rvert \le \lvert\theta\rvert \le \pi$
\begin{align}
    \label{eq:h_taylor_hayman_def}
    F(r\eul^{\iu\theta})\eul^{-n\iu\theta}
    = F(r)\cdot \e{\iu \theta (a(r)-n) - \frac{\theta^2}{2} b(r) + \frac{(\iu\theta)^3}{6} c(r\eul^{\iu\xi})}
\end{align}
for functions $a,b$ and $c$ given by
\begin{align*}
    \iu a(x)
    := \frac{\partial }{\partial \theta} f(x\eul^{\iu\theta})\bigg|_{\theta=0},
    \quad -b(x)
    := \frac{\partial^2}{\partial \theta^2}f(x\eul^{\iu\theta})\bigg|_{\theta=0}
    \quad\text{and}\quad 
    \iu^3 c(x) 
    := \frac{\partial^3}{\partial \theta^3}f(x\eul^{\iu\theta})\bigg|_{\theta=0}.
\end{align*}
In particular,
\begin{align}
    \label{eq:a,b,c_hayman_def}
    a(x)
    = xf'(x), \quad 
    b(x) = x^2f''(x) + xf'(x) 
    \quad\text{and}\quad 
    c(x) = x^3f'''(x) + 3x^2f''(x) + xf'(x).
\end{align}
With these definitions at hand we (informally) say that $F(x)$ is $H$-admissible, if it is possible to split up~\eqref{eq:cauchy_integral_hayman_def} into a dominant part, where $f(r\eul^{\iu\theta})-n\iu\theta = f(r) + \iu\theta (a(r)-n) - \theta^2 b(r)/2 + o(1)$, and another integral that is negligible. Then by choosing $a(r)$ to be (close to) $n$ the asymptotic value of the dominant integral can be retrieved, as it is of ``Gaussian'' type. The following three conditions formalize this idea, where $F$ is assumed to be a function with radius of convergence $0<\rho\le\infty$ which is positive on some interval $(R_0,\rho)\subseteq(0,\rho)$.
\begin{enumerate}[label=($H_{\arabic*}$)]
    \item~[Capture Condition]\label{item:hayman_capture} $a(r)$ and $b(r)$ tend to infinity as $r\to\rho$.
    \item~[Locality Condition]\label{item:hayman_locality} For some function $\theta_0:(R_0,\rho)\to \Real^+$ we have as $r\to\rho$ uniformly in $\lvert\theta\rvert \le \theta_0(r)$
    \[
        F(r\eul^{\iu\theta})
        \sim F(r) \cdot \e{\iu \theta a(r) - {\theta^2}b(r)/2}.
    \]
    \item~[Decay Condition]\label{item:hayman_decay} As $r\to\rho$ uniformly in $\theta_0(r) \le\lvert \theta\rvert <\pi$
    \[
        F(r\eul^{\iu\theta})
        = o\big( b(r)^{-1/2}{F(r)}\big).
    \]
\end{enumerate}
We call $F(x)$ $H$-admissible if it has the three proprties~\ref{item:hayman_capture}--\ref{item:hayman_decay}.
The following statement, which originates in~\cite[Thm.~1]{Hayman1956}, provides a useful tool for determining the $n$-th coefficient of a $H$-admissible function, see also~\cite[Prop.~VIII.5]{Flajolet2009}.
\begin{lemma}
\label{lem:hayman}
Suppose that $F(x)$ is $H$-admissible. Then as $r\to\rho$ 
\[
    [x^n]F(x)
    =  \frac{F(r)}{\sqrt{2\pi b(r)}} \cdot r^{-n} \cdot \left( \e{ -\frac{(a(r)-n)^2}{2b(r)}} + \eps_n\right),
\]
where $\lim_{r\to\rho} \sup_{n\in\Nat} \lvert\eps_n\rvert = 0 $.
\end{lemma}
In particular, by choosing any $r$ such that $(a(r)- n)^2/b(r) = \bigO{1}$ we get the first asymptotic order of $[x^n]F(x)$. This allows us to compare \emph{different} coefficients $[x^{n-k}]F(x)$ and $[x^n]F(x)$ using the \emph{identical} saddle-point $r$, a simple yet impactful fact we will use numerous times later.
\subsection{Asymptotics and Estimates for Power Series} 
\label{sec:power_series_hayman}

First we state the following auxiliary lemma which will help us to compute asymptotic bounds for the sum $\sum_{j\ge 1}j^{\beta-1} r^{\gamma j}C^{(\gamma)}((\eul^{-\chi})^j)$ for $\beta\in\Nat,\gamma\in\Nat_0$ and as $\chi\to 0$. In the proof we use the Euler-MacLaurin summation formula (and the computations are inspired by~\cite[Appendix A]{McKinley2020_preprint}). 
\begin{lemma}
\label{lem:sum_k^a/(1-x)^b_different_regimes_proof_hayman}
Let $\beta,\gamma\in\Real_0^+$. Then, as $\chi\to 0$,
\begin{align*}
    \sum_{k\ge 1} \frac{k^\gamma \eul^{-\chi k}}{(1-\eul^{-\chi k})^\beta} 
    \sim \begin{dcases*}
        d_1 \cdot \chi^{-(\gamma+1)}
        , &$\beta<1+\gamma$ \\ 
        \chi^{-(\gamma+1)}\ln( \chi^{-1}), & $\beta= 1 + \gamma $ \\
        d_2\cdot \chi^{-\beta}, & $\beta> 1 + \gamma $
    \end{dcases*},
\end{align*}
where, letting $\zeta$ denote the Zeta-function,
\begin{align*}
    d_1 = \int_0^\infty \frac{t^\gamma\eul^{-t}}{(1-\eul^{-t})^\beta} dt \quad\text{and} \quad 
    d_2 = \zeta(\beta-\gamma).
\end{align*}
\end{lemma}
\begin{proof}
Define $g(t) := t^\gamma\eul^{-\chi t}/(1-\eul^{-\chi t})^\beta$.
For $\beta-\gamma<1$ the integral $\int_0^\infty t^\gamma\eul^{-t}/(1-\eul^{-t})^\beta dt$ exists (since the integrand is asymptotically $t^{-(\beta-\gamma)}$ as $t\to 0$) so that by convergence of the Riemann sum to the corresponding integral we obtain
\[
    \sum_{k\ge 1} g(k)
    = \chi^{-(\gamma+1)} \cdot \chi \cdot \sum_{k\ge 1} \frac{(\chi k)^\gamma\eul^{-\chi k}}{(1-\eul^{-\chi k})^\beta}
    \sim \chi^{-(\gamma +1)} \int_0^\infty \frac{t^\gamma\eul^{-t}}{(1-\eul^{-t})^\beta}dt.
\]
Next we consider the case $\beta-\gamma> 1$. Let $P_1(x) = x-\floor{x}-1/2$. Then the Euler-Maclaurin formula, see for example~\cite[Ch.~9.5]{Graham1994}, gives us
\begin{align}
    \label{eq:euler_maclaurin_hayman}
    \sum_{k\ge 1}g(k) 
    = \int_1^\infty g(t)dt  + \frac{g(1)}{2} + \int_1^\infty g'(x) P_1(x)dx.
\end{align}
We will determine the first integral by using dominated convergence.
Note that for $t\ge 1$
\[
    \chi^\beta g(t) t^{\beta-\gamma} = (\chi t)^\beta \eul^{-\chi t}/(1-\eul^{-\chi t})^\beta.
\]
Thus
\begin{align}
    \label{eq:pointwise_convergence_hayman}
    \lim_{\chi\to 0}\chi^\beta g(t) = t^{\gamma-\beta},
    \quad
    t \ge 1.
\end{align}
Further, the continuous function $z^\beta\eul^{-z}/(1-\eul^{-z})^\beta$ tends to $1$ as $z\to 0$ and to $0$ as $z\to\infty$. Hence there is some $A>0$ such that \begin{align}
    \label{eq:upper_bound_dom_conv_hayman}
    \chi^\beta g(t) \le {A}t^{\gamma-\beta},
    \quad
    t\ge 1.
\end{align}
Thus, by dominated convergence,~\eqref{eq:pointwise_convergence_hayman}, and the fact that $\int_1^\infty t^{-(\beta-\gamma)}dt$ exists 
\begin{align}
    \label{eq:pointwise_convergence_hayman_next_step}
    \int_1^\infty g(t)dt = \chi^{-\beta}  \int_1^\infty \chi^\beta g(t) dt 
    \sim \chi^{-\beta} \int_1^\infty t^{-(\beta-\gamma)} dt
    = \chi^{-\beta}\frac{1}{\beta-\gamma-1}.
\end{align}
The next term in~\eqref{eq:euler_maclaurin_hayman} is $g(1)/2 \sim \chi^{-\beta}/2$. Moreover,
\begin{align}
    \label{eq:int_g'_P_1_hayman}
    \int_1^\infty g'(t) P_1(t)dt
    = \int_1^\infty \left( \gamma\frac{t^{\gamma-1}\eul^{-\chi t}}{(1-\eul^{-\chi t})^\beta} 
    - \chi\frac{t^{\gamma}\eul^{-\chi t}}{(1-\eul^{-\chi t})^\beta}
    - \chi\beta \frac{t^{\gamma}\eul^{-2\chi t}}{(1-\eul^{-\chi t})^{\beta+1}}\right)P_1(t) dt.
\end{align}
Note that $\beta-\gamma>1$ implies that $\beta-(\gamma-1)>1$ and $(\beta+1)-\gamma>1$ so that as before
\[
    \int_1^\infty g'(t) P_1(t)dt
    \sim \chi^{-\beta}(\gamma-\beta) \int_1^\infty t^{-(\beta-\gamma +1)} P_1(t)dt 
    - \chi^{-\beta+1} \int_1^\infty t^{-(\beta-\gamma)}P_1(t)dt.
\]
As $\beta - \gamma > 1$, the last term is $\bigO{\chi^{-\beta+1}} = o(\chi^{-\beta})$. Moreover, by using again Euler-Maclaurin summation
\[
    \sum_{k\ge 1}\frac{k^{-(\beta-\gamma)}}{\gamma-\beta} 
    = \int_1^\infty \frac{t^{-(\beta-\gamma)}}{\gamma-\beta} dt + \frac{1}{2(\gamma-\beta)} 
    + \int_1^\infty t^{-(\beta-\gamma+1)}P_1(t)dt.
\]
Computing the integral and rearranging the terms yields
\[
    \int_1^\infty t^{-(\beta-\gamma+1)}P_1(t)dt
    = \frac{\zeta(\beta-\gamma)}{\gamma-\beta} + \frac1{(\gamma-\beta)(\gamma-\beta+1)} -\frac{1}{2(\gamma-\beta)} .
\]
The claim follows for $\beta-\gamma>1$. 
Finally, let us consider the case $\beta = \gamma + 1$. Set $a = \ln(\chi^{-1})^{-1/(2(\gamma+1))} = o(1)$. We have that
\[
    \chi^{\gamma+1}\int_1^\infty g(t) dt
    = \int_\chi^\infty \frac{t^\gamma \eul^{-t}}{(1-\eul^{-t})^{\gamma+1}} dt
    =\left( \int_{\chi}^a + \int_a^1 + \int_1^\infty\right) \frac{t^\gamma \eul^{-t}}{(1-\eul^{-t})^{\gamma+1}} dt
    =: I_1 + I_2 + \bigO{1}.
\]
In $I_1$ we use that $t=o(1)$ to obtain that $I_1 \sim \int_{\chi}^a t^{-1} dt = \ln a - \ln \chi \sim \ln(\chi^{-1})$. In $I_2$ we estimate $I_2 \le (1-\eul^{-a})^{-(\gamma+1)}\int_a^1 t^\gamma\eul^{-t} = \Theta( \ln(\chi^{-1})^{1/2}) =o(\ln(\chi^{-1}))$.
Hence
\begin{align}
    \label{eq:gamma+1=beta_case_ln_int}
    \int_1^\infty g(t)dt \sim \chi^{-(\gamma+1)}\ln(\chi^{-1}).
\end{align}
Further $g(1) \sim \chi^{-\beta} = \chi^{-(\gamma+1)}$ and by estimating $\lvert P_1(x)\rvert \le 1$ in~\eqref{eq:int_g'_P_1_hayman} we get 
\[
    \bigg\lvert \int_1^\infty g'(t)P_1(t)dt \bigg\rvert
    \le \int_1^\infty \left( \gamma \frac{t^{\gamma-1}\eul^{-\chi t}}{(1-\eul^{-\chi })^\beta} 
    + \chi \frac{t^\gamma \eul^{-\chi t}}{(1-\eul^{-\chi t})^\beta} + \chi \beta \frac{t^\gamma \eul^{-2\chi t}}{(1-\eul^{-\chi t})^{\beta+1}}
    \right)dt
    =: J_1 + J_2 + J_3.
\]
Since $\beta = \gamma +1 $ we have that $\gamma-1-\beta = -2$ and the integral $\int_1^\infty t^{-2}dt$ exists. Hence, analogous to~\eqref{eq:pointwise_convergence_hayman}--\eqref{eq:pointwise_convergence_hayman_next_step} we obtain by dominated convergence that $J_i = \bigO{\chi^{-(\gamma+1)}}=o(\chi^{-(\gamma+1)}\ln(\chi^{-1}))$ for $i=1,3$. Analogous to~\eqref{eq:gamma+1=beta_case_ln_int} we obtain that
\[
    \chi^{-1}  J_2
    =  \int_1^\infty g(t)dt
    \sim\chi^{-(\gamma+1)}\ln(\chi^{-1})
\]
implying that $J_2 = o(\chi^{-(\gamma+1)}\ln(\chi^{-1})).$
This finishes the proof.
\end{proof}
In the following two lemmas we consider an eventually positive, slowly varying function $h:\Real^+\to\Real_0^+$ which means that $\lim_{x\to\infty}h(\lambda x)/h(x)=1$ for any $\lambda>0$. An immediate consequence from~\cite[Theorem 1.3.1]{Bingham1987} for such functions is that they are subpolynomial, that is, for any $\delta>0$ and $x_0$ sufficiently large
\begin{align}
    \label{eq:sv_estimate_hayman}
	x^{-\delta} 
	\le h(x) 
	\le x^\delta
	\quad\text{and}\quad 
	\Big(\frac{x}{x'}\Big)^{\delta}  
	\le \frac{h(x')}{h(x)} 
	\le \Big(\frac{x'}{x}\Big)^{\delta} 
	\quad \text{ for all } 
	x' \ge x \ge {x}_0.
\end{align}
The next result is a consequence of the famous \emph{Karamata's Tauberian Theorem} which was originally derived in~\cite{Karamata1931}. 
\begin{lemma}[{\cite[Thm.~1.7.1]{Bingham1987}}]
\label{lem:asymptotics_expansive_sum}
Let $\al  > 0$ and $h$ be an eventually positive and slowly varying function. Then, as $\chi \to0$
\[
    \sum_{k\ge 1} h(k)k^{\al-1} \eul^{-\chi k}
    \sim \Gamma(\al) h(\chi^{-1})\chi^{-\al}.
\]
\end{lemma}
The following lemma shows that slowly varying functions appearing in integrals can often be handled with quite easily.
\begin{lemma}[{\cite[Prop. 2]{Freiman2005}}]
\label{lem:asymptotics_expansive_integral_arb_start}
Let $\al>0$ and $h$ be an eventually positive and slowly varying function. Let $(b_n)_{n\in\Nat},(t_n)_{n\in\Nat}$ be sequences such that $b_n\to b\in(0,\infty]$ and $b_nt_n\to\infty$ as $n\to\infty$. Then, as $n\to\infty$,
\[
    \int_{b_n}^\infty h(xt_n)x^{\al}\eul^{-x}dx
    \sim h(b_nt_n) \int_{b_n}^\infty x^{\al}\eul^{-x}dx.
\]
\end{lemma}
The last lemma in this subsection is a simple estimate for series with non-negative coefficients and a positive radius of convergence.
\begin{lemma}[{\cite[Lem.~2.7]{Panagiotou2022}}]
\label{lem:C(x)_approx}
Let $m\in\Nat$. Let $F(x) = \sum_{k\ge m}f_kx^k$ be a power series with non-negative coefficients such that $f_m>0$ and radius of convergence $\rho>0$. Let $0<\eps<1$. Then there exists $A>0$ such that for all $0 \le z\le (1-\eps)\rho$
\[
	1 \le \frac{F(z)}{f_m z^{m}} \le 1 + A z.
\]
\end{lemma}

\section{Proofs}
\label{sec:proofs}
Here we collect the proofs of all our main results. In Section~\ref{sec:proof_h_admissibility} we state and prove Lemma~\ref{lem:h_admissible_gen_series}. Section~\ref{sec:proof_main} contains all the proofs of the statements in Section~\ref{sec:main_results}.

\subsection{Proof of $H$-admissibility}
\label{sec:proof_h_admissibility}
In this section we prove the following lemma, which is the backbone of the other forthcoming proofs in this section, and on the way some properties of functions related to~\eqref{eq:a,b,c_hayman_def}. 
\begin{lemma}
\label{lem:h_admissible_gen_series}
For $\al>0$, $0<\rho\le 1$ and $0<\eps<\al/3$ let $(c_k)_{k\in\Nat}\in\F(2\al/3+\eps,\al,\rho)$. Then
\[
    S(x)C(x)^\ell \quad \text{and}\quad  G(x) \prod_{1\le i\le \ell}\sum\nolimits_{j\ge 1}j^{p_i}C(x^j)
\]
are $H$-admissible for $\ell\in\Nat_0$ and $(p_1,\dots,p_\ell)\in\Nat_0^\ell$.
\end{lemma}

\subsubsection{Proof of Lemma~\ref{lem:h_admissible_gen_series} in the Set Case} 
\label{sec:proof_lemma_H-adm_set_case}
In this section we prove that $F_\ell(x) = S(x)C(x)^\ell$ is $H$-admissible under the conditions stated in Lemma~\ref{lem:h_admissible_gen_series}, that is, we verify properties~\ref{item:hayman_capture}--\ref{item:hayman_decay} for $F_\ell$ and the corresponding
$$f_\ell = \ln F_\ell = C + \ell \ln C.$$ Before we do so, as a preparation we establish some useful properties of the functions defined in~\eqref{eq:a,b,c_hayman_def}
\begin{align}
\label{eq:def_a_ell}
    a_\ell(x)
    := xf_\ell'(x), 
    \quad b_\ell(x)
    := x^2f_\ell''(x)+xf_\ell'(x)
    \quad\text{and}\quad 
    c_\ell(x)
    := x^3f_\ell '''(x) + 3x^2f_\ell''(x) + xf_\ell'(x) \enspace .
\end{align}
To simplify the notation we also introduce
\begin{align}
    \label{eq:A_s_def_hayman_proof}
    A_s(x)
    := \sum_{k\ge 1}k^sc_kx^k ,\quad s\in\Nat_0.
\end{align}
With this definition at hand we are able to abbreviate
\begin{align*}
    A_0(x) = C(x),
    ~~ A_1(x) = xC'(x),
    ~~ A_2(x) = x^2C''(x)+xC'(x),
    ~~ A_3(x) = x^3C'''(x) + 3x^2C''(x) + xC'(x).
\end{align*}
Then we obtain 
\begin{align}
\begin{split}
    \label{eq:def_a,b,c_set_ell_hayman_proof}
    a_\ell(x)
    &= A_1(x) + \ell \frac{A_1(x)}{A_0(x)}, \quad 
    b_\ell(x)
    = A_2(x) + \ell \bigg(\frac{A_2(x)}{A_0(x)} - \frac{A_1(x)^2}{A_0(x)^2}\bigg) 
    \quad \text{and} \\
    c_\ell(x) &= A_3(x) + \ell \bigg( \frac{A_3(x)}{A_0(x)} - 3\frac{A_1(x)A_2(x)}{A_0(x)^2} + 2\frac{A_1(x)^3}{A_0(x)^3}\bigg).
    \end{split}
\end{align}
An immediate consequence of Lemma~\ref{lem:asymptotics_expansive_sum} is the following auxiliary statement.
\begin{corollary}
\label{lem:A_s_properties_hayman_proof}
Let $(c_k)_{k\in\Nat}\in\F (2\al/3+\eps,\al,\rho)$ for $0<\eps<\al/3, 0<\rho\le 1$ and $\al>0$. Set $r=\rho\eul^{-\chi}$ for $\chi>0$. Then, for any $s\in\Nat_0$ 
\begin{align}
    \label{eq:A_s_2bigOmega_hayman_proof}
    A_s(r)
    = {\cal O}\big(\chi^{-(\al+s)}\big)
    \quad\text{and}\quad 
    A_s(r) 
    =\Omega( \chi^{-(2\al/3+\eps + s)}),\quad \text{ as } \chi\to 0.
\end{align}
\end{corollary}
In the next statement we determine simple(-r) asymptotic expressions for the functions in~\eqref{eq:def_a,b,c_set_ell_hayman_proof} when the argument gets close to the radius of convergence.
\begin{lemma}
 \label{lem:a_l_b_l_hayman_proof}
Let $(c_k)_{k\in\Nat}\in\F (2\al/3+\eps,\al,\rho)$ for $0<\eps<\al/3, 0<\rho\le 1$ and $\al>0$. Set $r=\rho\eul^{-\chi}$ for $\chi>0$. Then the functions in~\eqref{eq:def_a,b,c_set_ell_hayman_proof} fulfill for any $\ell\in\Nat_0$
\begin{align}
    \label{eq:a_l_asympt}
    a_\ell(r) \sim A_1(r) \quad \text{and}\quad  b_\ell(r) \sim A_2(r), \quad\text{as }\chi\to 0.
\end{align}
\end{lemma}
\begin{proof}
We obtain that $A_0(r) \to\infty$ by~\eqref{eq:A_s_2bigOmega_hayman_proof}, and so  $a_\ell(r) = A_1(r) + \bigO{A_1(r)/A_0(r)} \sim A_1(r)$ as $\chi\to 0$. Further, by applying Hölder's inequality
\begin{align}
    \label{eq:hoelder_ineq_applied_hayman_proof}
    A_1(r)^2
    = \bigg(\sum_{k\ge 1}k\sqrt{c_kr^k} \cdot \sqrt{c_kr^k} \bigg)^2 
    \le A_2(r) A_0(r).
\end{align}
Hence $(A_1(r)/A_0(r))^2 \le A_2(r)/A_0(r)$, and the second statement follows by using $A_0(r)\to\infty$ once again.
\end{proof}
Now we are able to proceed with the proof of Lemma~\ref{lem:h_admissible_gen_series}.
\begin{proof}[Proof of Lemma~\ref{lem:h_admissible_gen_series} (set case).]
We verify properties~\ref{item:hayman_capture}--\ref{item:hayman_decay} for $a_\ell(x), b_\ell(x)$ and $c_\ell(x)$ from~\eqref{eq:def_a,b,c_set_ell_hayman_proof}. 
Set $r=\rho\eul^{-\chi}$. 
From~\eqref{eq:a_l_asympt} we obtain that $a_\ell(r)\sim A_1(r)$ and $b_\ell(r)\sim A_2(r)$. Then~\eqref{eq:A_s_2bigOmega_hayman_proof} implies that $a_\ell(r)$ and $b_\ell(r)$ both tend to infinity as $\chi\to 0$, thus establishing~\ref{item:hayman_capture}. 

We prove~\ref{item:hayman_locality}. For some $\al/3<\delta<\al/3+\eps/2$ set
\begin{align}
    \label{eq:theta_0_def_hayman_proof}
    \theta_0
    := \chi^{1+\delta}.
\end{align}
By applying Taylor's expansion we obtain that for $\lvert\theta\rvert \le \theta_0$ there is a $\xi=\xi(\theta)\in(-\theta,\theta)$ such that
\[
    F_\ell(r\eul^{\iu\theta})
    =F_\ell(r) \cdot \e{ \iu\theta a_\ell(r) - \frac{\theta^2}{2}b_\ell(r) +\iu^3 \frac{\theta^3}{6} c_\ell(r\eul^{\iu\xi})}.
\]
Defining 
\[
    d_\ell(x)
    := \ell\left(\frac{A_3(x)}{A_0(x)} - 3\frac{A_1(x)A_2(x)}{A_0(x)^2} + 2 \frac{A_1(x)^3}{A_0(x)^3}\right)
\]
we get in view of~\eqref{eq:def_a,b,c_set_ell_hayman_proof} that $\lvert \theta^3 c_\ell(r\eul^{\iu\xi})\rvert \le \theta_0^3\lvert A_3(r\eul^{\iu\xi})\rvert + \theta_0^3\lvert d_\ell(r\eul^{\iu\xi})\rvert$ uniformly in $\lvert\theta\rvert\le\theta_0$.
Since $A_3(r)$ has only non-negative coefficients,  the triangle inequality asserts that $\lvert A_3(r\eul^{\iu\xi})\rvert \le A_3(r)$. With~\eqref{eq:A_s_2bigOmega_hayman_proof} we further get that $A_3(r) = \bigO{ \chi^{-(\al+3)} }$ as $\chi\to0$.
We conclude that $\theta_0^3\lvert A_3(r\eul^{\iu\xi})\rvert \le \theta_0^3 A_3(r) = \bigO{\chi^{-\al+3\delta}} = o(1)$ as $\chi\to 0$ uniformly in $\lvert\theta\rvert\le \theta_0$. It remains to show that $\theta_0^3\lvert d_\ell(r\eul^{\iu\xi})\rvert = o(1)$ in order to get~\ref{item:hayman_locality}.
Since the functions appearing in $d_\ell$ cannot necessarily be represented as power series with non-negative coefficients (powers of $A_0$ appear in the denominator), we cannot use the triangle inequality to get rid of $\eul^{\iu\xi}$ that easily. So, we first determine a lower bound for $\lvert A_0(r\eul^{\iu\xi})\rvert$ that holds uniformly in $\lvert\theta\rvert \le \theta_0$; in fact we are going to show that $\lvert A_0(r\eul^{\iu\xi})\rvert \sim A_0(r)$. From the definition of the absolute value of complex numbers 
\begin{align} 
    \label{eq:A_0_complex_absolute_value_hayman_proof}
    \lvert A_0(r\eul^{\iu\xi}) \rvert^2 
    = \bigg(\sum_{k\ge 1}c_kr^k\cos(\xi k)\bigg)^2 + \bigg(\sum_{k\ge 1}c_kr^k\sin(\xi k)\bigg)^2
    \ge \bigg( \sum_{k\ge 1}c_kr^k\cos(\xi k) \bigg)^2.
\end{align}
For $1\le k\le \chi^{-1-\delta/2}$ we have $k\xi=o(1)$ since $\lvert \xi\rvert\le\lvert\theta\rvert \le \theta_0=\chi^{1+\delta}$. Hence
\[
    \sum_{k\ge 1}c_kr^k\cos(\xi k)
    \sim A_0(r) - R_0 + R_1, \quad R_i = \sum_{k>\chi^{-1-\delta/2}} c_kr^k \cos(\xi k)^i, \,i=0,1.
\]
Then for $i=0,1$ and recalling that $(c_k)_{k\in\Nat}\in \F(2\al/3+\eps,\al,\rho)$
\[
    \lvert R_i\rvert 
    = \mathcal{O}\bigg( \sum_{k> \chi^{-1-\delta/2}} k^{\al-1} \eul^{-\chi k}\bigg)
    = \Theta\left( \chi^{-\al} \int_{\chi^{-\delta/2}}^\infty t^{\al-1}\eul^{-t}dt\right)
    =o(1).
\]
Together with $\lvert A_0(r\eul^{\iu\xi})\rvert \le A_0(r)$ we thus deduce $\lvert A_0(r\eul^{\iu\xi})\rvert \sim A_0(r)$ as $\chi\to 0$ uniformly in $\lvert \xi\rvert \le\lvert\theta\rvert\le\theta_0$. Using that $A_s$ has only non-negative coefficients for $s\ge 0$ we then get the estimate
\[
    \ell^{-1}\lvert d_\ell(r\eul^{\iu\xi}) \rvert 
    \le \frac{A_3(r)}{\lvert A_0(r\eul^{\iu\xi})\rvert} + 3\frac{A_1(r)A_2(r)}{\lvert A_0(r\eul^{\iu\xi})\rvert^2} + 2\frac{A_1(r)^3}{\lvert A_0(r\eul^{\iu\xi})\rvert^3}
    \sim \frac{A_3(r)}{A_0(r)} + 3\frac{A_1(r)A_2(r)}{A_0(r)^2} + 2\frac{A_1(r)^3}{A_0(r)^3}.
\]
With H{\"o}lder's inequality we obtain
\begin{align*}
     A_2(r)^2 
     &= \bigg( \sum_{k\ge 1}k^{3/2} \sqrt{c_k r^k} \cdot \sqrt{kc_kr^k} \bigg)^2
    \le A_3(r)A_1(r) \quad\text{and}  \\
    A_1(r)^3 
    &= \bigg(\sum_{k\ge 1} k (c_kr^k)^{1/3} \cdot (c_kr^k)^{2/3} \bigg)^3
    \le A_3(r) A_0(r)^2.
\end{align*}
In~\eqref{eq:hoelder_ineq_applied_hayman_proof} we showed that $A_1(r)^2 \le A_2(r)A_0(r)$. With this at hand, we get that $A_1(r)A_2(r)/A_0(r)^2 \le A_2(r)^2/(A_0(r)A_1(r)) \le A_3(r)/A_0(r)$ and $A_1(r)^3/A_0(r)^3 \le A_3(r)/A_0(r)$. Since $A_0(r)\to\infty$ according to~\eqref{eq:A_s_2bigOmega_hayman_proof} this implies with~\eqref{eq:A_s_2bigOmega_hayman_proof} that, as $\chi\to 0$ and uniformly in $\lvert \theta\rvert \le \theta_0$,
\[
    \theta_0^3\lvert d_\ell(r\eul^{\iu\xi})\rvert 
    = \bigO{ \theta_0^3\frac{A_3(r)}{A_0(r)} }
    = o(\theta_0^3A_3(r))
    =o(1).
\]
This delivers~\ref{item:hayman_locality}. The hardest part of the proof is to show~\ref{item:hayman_decay}, but, as already mentioned in the introduction,~\cite[Lem.~7]{Freiman2002} solves the problem in an almost identical setting; we repeat (and sometimes adapt) the arguments for completeness. Since $C(x)$ has only non-negative coefficients 
\begin{align}
    \label{eq:hayman_decay_proof_1}
    \bigg\lvert\frac{F_\ell(r\eul^{\iu\theta})}{F_\ell(r)}\bigg\rvert
    \le \big\lvert e^{C(r\eul^{\iu\theta})-C(r)} \big\rvert 
    \le\exp\Big\{ \sum_{k\ge 1}c_kr^k (\cos(\theta k)-1) \Big\}
    = \exp\Big\{-2 \sum_{k\ge 1} c_k r^k \sin^2(\theta k/2)\Big\}.
\end{align}
Since $\sin^2(x)$ is symmetric, it is no restriction to consider $\theta>0$ from now on. Denote by $\lVert x\rVert$ the distance from $x\in\Real$ to its nearest integer. We will use the well-known fact
\[
    \sin^2(\pi x)
    \ge 4\lVert x\rVert.
\]
Setting $\al_1 := 2\al/3 + \eps$
and letting $A_1>0$ be such that $c_k\ge A_2k^{\al_1-1}\rho^{-k}$ for $k\in\Nat$, we get that
\[
    (4A_1)^{-1} \sum_{k\ge 1} c_k r^k \sin^2(\theta k/2)
    \ge \sum_{k\ge 1} k^{\al_1-1}\eul^{-\chi k} \lVert \theta k/(2\pi)\rVert 
    =: V(\theta/(2\pi)).    
\]
We claim that there exists $f>0$ such that, as $\chi\to 0$ and uniformly in $\theta_0/(2\pi) \le \theta < 1/2$,
\begin{align}
    \label{eq:hayman_decay_proof_2}
    V(\theta)
    \ge \chi^{-f}.
\end{align}
Since $b_\ell(r) \sim A_2(r)=\bigO{ \chi^{-(\al+2)}}$ due to~\eqref{eq:A_s_2bigOmega_hayman_proof} we obtain that $b(r)^{-1/2} = \Omega( \chi^{-(\al/2+1)})$ implying with~\eqref{eq:hayman_decay_proof_1} that as $\chi\to0$ and uniformly in $\theta_0\le \theta < \pi$
\[
     \bigg\lvert\frac{F_\ell(r\eul^{\iu\theta})}{F_\ell(r)}\bigg\rvert
     \le \e{-8A_1 V(\theta)}
     \le \e{-8A_1\chi^{-f}}
     =o(b(r)^{-1/2}),
\]
which is Condition~\ref{item:hayman_decay}. So, if we show~\eqref{eq:hayman_decay_proof_2} we are done.

From here we basically copy the proof of~\cite[Lem. 7]{Freiman2002} with minor adaptions. 
We partition $[\theta_0/2\pi,1/2)$ into $I_1:=[\theta_0/(2\pi),\chi]$ and $I_2:=(\chi,1/2)$. Consider $\theta\in I_1$. For such $\theta$ we have that
\[
    \lVert \theta k\rVert = \theta k,
    \quad k\le (2\chi)^{-1},
\]
implying
\begin{align}
    \label{eq:hayman_decay_proof_3}
    V(\theta)
    \ge \theta_0^2/(2\pi)^2 \sum_{1\le k\le (2\chi)^{-1}}
    k^{\al_1+1}\eul^{-\chi k}
    = \Theta(\chi^{2\delta-\al_1}).
\end{align}
Since $\delta$ is chosen such that 2$\delta-\al_1 =2 \delta - 2\al/3 - \eps <-\al/3-\eps/2<0$ the claim~\eqref{eq:hayman_decay_proof_2} follows as $\chi\to 0$ and uniformly in $\theta\in I_1$.  Let us now consider $\theta\in I_2$. Define the sets
\[
    Q(\theta)
    := \{ k\ge 1: \lVert \theta k\rVert \ge 1/4\}
    = \bigcup_{j\ge 0} Q_j(\theta),
    \quad Q_j(\theta) := \{k\ge 1: j+1/4 \le \theta k \le j+3/4\}.
\]
Then
\[
    16\cdot V(\theta)
    \ge \sum_{k\in Q(\theta)} k^{\al_1-1}\eul^{-\chi k}
    = \sum_{j\ge 0} \sum_{k\in Q_j(\theta)} k^{\al_1-1}\eul^{-\chi k}.
\]
The intuition behind the choice of these sets is that for any $\theta\in I_2$ and $j\ge 0$ there is a least one element of order $j/\theta$ in $Q_j(\theta)$ because $(j+3/4)/\theta - (j+1/4)/\theta =(2\theta)^{-1}\ge 1$. In particular for $j$ close to $\chi^{-1}\theta$ we sum over at least one $k$ with magnitude $\chi^{-1}$, the range where $k^{\al_1-1}\eul^{-\chi k}$ contributes the most to the entire sum so that the asymptotic order of $\sum_{k\ge 1}k^{\al_1-1}\eul^{-\chi k} = \Theta(\chi^{-\al_1})$ can be recovered even in the limited range $k\in Q(\theta)$. At the same time the term $\lVert\theta k\rVert$ is bounded from below uniformly in $k\in Q(\theta)$ and $\theta\in I_2$.

To substantiate this, as a next step we estimate the sum by an integral. Since for $j\ge 0$ we have that $(j+1/4)/\theta \le k \le (j+3/4)/\theta$ we deduce for $\chi j/\theta$ sufficiently large, say $\ge s_0>0$, that $(\chi k)^{\al-1}\eul^{-\chi k}$ is monotonic decreasing for $k\in \bigcup_{j\ge s_0 \theta\chi^{-1}} Q_j(\theta)$. Hence we find the following lower bound for the sum which holds as $\chi\to0$ uniformly in $\theta\in I_2$
\begin{align*}
    16\cdot V(\theta)
    &\ge \int_{s_0 \theta\chi^{-1}}^\infty \int_{(u+1/4)/\theta}^{(u+3/4)/\theta} v^{\al_1 - 1} \eul^{-\chi v} dv du 
    = \chi^{-\al_1}  \int_{s_0 \theta\chi^{-1}}^\infty  \int_{\chi(u+1/4)/\theta}^{\chi(u+3/4)/\theta} t^{\al_1 - 1} \eul^{- t} dt dv \\
    &= \chi^{-\al_1}  \frac{\theta}{\chi}  \int_{s_0}^\infty
    \int_{s + \chi/(4\theta)}^{s + 3\chi/(4\theta) } t^{\al_1 - 1} \eul^{- t} dt ds,
\end{align*}
where we first applied the change of variables $t=\chi v$ and then $s = \chi/\theta\cdot u$. Estimating $3\chi/(4\theta) \le 1$ and using again that the involved functions are decreasing in the range we are integrating over we get as $\chi\to 0$ uniformly in $\theta\in I_2$
\begin{align*}
    16\cdot V(\theta)
    \ge \chi^{-\al_1}  \frac{\theta}{\chi}  \int_{s_0}^\infty 
    \int_{s + \chi/(4\theta)}^{s + 3\chi/(4\theta) } t^{\al_1 - 1} \eul^{- t} dt ds 
    \ge \chi^{-\al_1}/2 \int_{t_0}^\infty (s+1)^{\al-1}\eul^{-(s+1)} ds 
    =\Theta(\chi^{-\al_1}).
\end{align*}
This shows together with~\eqref{eq:hayman_decay_proof_3} that~\eqref{eq:hayman_decay_proof_2} is valid as $\chi\to 0$ and uniformly in $\theta\in I_1\cup I_2$. Altogether, we have just established that $F_\ell$ has~\ref{item:hayman_decay} and the proof is completed.
\end{proof}

\subsubsection{Proof of Lemma~\ref{lem:h_admissible_gen_series} in the Multiset Case}  
In this section we prove that $\tilde{F}_\ell(x):=G(x)\prod_{1\le i\le \ell}\sum_{j\ge1}j^{p_i}C(x^j)$ is $H$-admissible under the conditions stated in Lemma~\ref{lem:h_admissible_gen_series}, that is, we verify properties~\ref{item:hayman_capture}--\ref{item:hayman_decay} for $\tilde{F}_\ell$ and the corresponding
$$
    \tilde{f}_\ell(x) := \ln \tilde{F}_\ell(x) = \sum_{j\ge 1}C(x^j)/j + \sum_{1\le i\le \ell} \ln \left(\sum_{j\ge1} j^{p_i}C(x^j)\right).
$$
Before we do so, as a preparation we establish some useful properties of the functions defined in~\eqref{eq:a,b,c_hayman_def}
\begin{align}
\label{eq:def_tilde_a_ell}
    \tilde{a}_\ell(x)
    := x\tilde{f}_\ell'(x), 
    \quad \tilde{b}_\ell(x)
    := x^2\tilde{f}_\ell''(x)+x\tilde{f}_\ell'(x)
    \quad\text{and}\quad 
    \tilde{c}_\ell(x)
    := x^3\tilde{f}_\ell '''(x) + 3x^2\tilde{f}_\ell''(x) + x\tilde{f}_\ell'(x).
\end{align}
To simplify the notation we introduce
\begin{align}
    \label{eq:A_s_t_def_hayman_proof}
    A_{s,t}(x)
    := \sum_{j\ge 1}j^{t-1} \sum_{k\ge 1}k^sc_kx^{jk}, 
    \qquad s,t\in\Nat_0.
\end{align}
With this notation
\begin{align*}
    A_{0,t}(x) &= \sum_{j\ge 1}j^{t-1}C(x^j), \quad A_{1,t}(x) = \sum_{j\ge 1}j^{t-1}x^jC'(x^j), \\
    A_{2,t}(x) &= \sum_{j\ge 1} j^{t-1}(x^{2j}C''(x^j)+x^jC'(x^j)) \quad\text{and} \\
    A_{3,t}(x) &= \sum_{j\ge 1}j^{t-1}(x^{3j}C'''(x^j) + 3x^{2j}C''(x^j) + x^jC'(x^j)).
\end{align*}
With this at hand, we can write
\begin{align}
    \begin{split}
    \label{eq:tilde_a_b_ell_hayman_proof}
    \tilde{a}_\ell(x) 
    &= A_{1,1}(x) + \sum_{1\le i\le \ell} \frac{A_{1,2+p_i}(x)}{A_{0,1+p_i}(x)}, \quad 
    \tilde{b}_\ell(x)
    = A_{2,2}(x) + \sum_{1\le i\le \ell} \bigg (\frac{A_{2,3+p_i}(x)}{A_{0,1+p_i}(x)} - \frac{A_{1,2+p_i}(x)^2}{A_{0,1+p_i}(x)^2} \bigg) \quad\text{and} \\
    \tilde{c}_\ell(x) 
    &= A_{3,3}(x) + \sum_{1\le i\le \ell} \bigg( \frac{A_{3,4+p_i}(x)}{A_{0,1+p_i}(x)}-3\frac{A_{1,2+p_i}(x)A_{2,3+p_i}(x)}{A_{0,1+p_i}(x)^2}+ 2\frac{A_{1,2+p_i}(x)^3}{A_{0,1+p_i}(x)^3}\bigg).
    \end{split}
\end{align}
We state some asymptotic properties of~\eqref{eq:A_s_t_def_hayman_proof} based on Lemma~\ref{lem:sum_k^a/(1-x)^b_different_regimes_proof_hayman}. For this also recall the definition of $A_s(x)$ from~\eqref{eq:A_s_def_hayman_proof}.
\begin{lemma}
\label{lem:A_s_t_properties_hayman_proof}
Let $(c_k)_{k\in\Nat}\in\F (2\al/3+\eps,\al,\rho)$ for $0<\eps<\al/3, 0<\rho\le 1$ and $\al>0$. Set $r=\rho\eul^{-\chi}$ for $\chi>0$. Then the following statements are true for any $s,t\in\Nat_0$ and as $\chi\to 0$. First,
\begin{align}
    \label{eq:A_s_s+1_o(1)_hayman_proof}
    \frac{A_{s,s+t}(r)}{A_{s,s}(r)A_{0,t}(r)}
    =o(1).
\end{align}
Moreover, if $0 < \rho < 1$, then
\begin{align}
    \label{eq:A_s_t_0<rho<1_proof_hayman}
    A_{s,t}(r)
    = A_s(r) + \sum_{j\ge 2}j^{t-1} \sum_{k\ge 1}k^sc_k\rho^{js} + o(1)
    = A_s(r) + \bigO{1}.
\end{align}
Finally, if $\rho = 1$, then
\begin{align}
    \label{eq:A_s_t_bigO_proof_hayman}
    A_{s,t}(r)
    = \begin{dcases*}
        {\cal O}(\chi^{-(\al+s)}) , & $t<\al+s$\\
        {\cal O}(\chi^{-(\al+s)} \ln( \chi^{-1})), & $t = \al+s$ \\
        {\cal O}(\chi^{-t}), & $t > \al+s$
    \end{dcases*},
    \quad
    A_{s,t}(r)
        = \begin{dcases*}
        \Omega(\chi^{-(2\al/3+\eps+s)}) , & $t<2\al/3+\eps+s$\\
        \Omega(\chi^{-(2\al/3+\eps+s)} \ln( \chi^{-1}) ), & $t = 2\al/3+\eps+s$ \\
        \Omega(\chi^{-t}), & $t > 2\al/3+\eps+s$
    \end{dcases*}.
\end{align}
\end{lemma}
\begin{proof}
To show~\eqref{eq:A_s_t_0<rho<1_proof_hayman} we first note that $A_{s,t}(r) = A_s(r) + \sum_{j\ge 2}j^{t-1}\sum_{k\ge 1}k^sc_kr^{jk}$. Note that for $0<\rho<1$ we have some $\eps>0$ such that for all $j\ge 2$ it holds that $r^{jk} \le (1-\eps)^k \rho^k$. Hence applying Lemma~\ref{lem:C(x)_approx} shows that for $0<\rho<1$
\[
    \sum_{j\ge 2}j^{t-1}\sum_{k\ge 1}k^sc_kr^{jk}
    = \bigO{\sum_{j\ge 2}j^{t-1}\rho^{j}}
    =\bigO{1}.
\] 
By dominated convergence we can let $r\to \rho$. 
In addition, if $0<\rho<1$, then due to~\eqref{eq:A_s_t_0<rho<1_proof_hayman} we obtain $A_{s,s+t}(r)/(A_{s,s}(r)A_{0,t}(r)) \sim A_{0,t}(r)^{-1}=o(1)$. This shows all the statements for $0<\rho<1$.

For the remaining proof assume $\rho=1$. Then for $s,t\in\Nat_0$ as $\chi\to 0$
\begin{align*}
    A_{s,t}(r)
    = \sum_{k\ge 1}k^s c_k \sum_{j\ge 1}j^{t-1} \eul^{-\chi k j}
    = \Theta \bigg( \sum_{k\ge 1} \frac{k^s c_k \eul^{-\chi k}}{(1-\eul^{-\chi k})^{t}} \bigg).
\end{align*}
With this at hand, Lemma~\ref{lem:sum_k^a/(1-x)^b_different_regimes_proof_hayman} reveals~\eqref{eq:A_s_t_bigO_proof_hayman}
.
In turn, with~\eqref{eq:A_s_t_bigO_proof_hayman}
we compute for $s,t\in\Nat$
\begin{align*}
    \frac{A_{s,s+t}(r)}{A_{s,s}(r)A_{0,t}(r)}
    = \begin{dcases*}
         \bigO{\chi^{2\al/3 + \eps }\ln(\chi^{-1})}   ,& $2\al/3+\eps<\al\le t$ \\
         \bigO{\chi^{-\al/3 +\eps +t}}   ,& $ 2\al/3+\eps \le t < \al$ \\
         \bigO{\chi^{\al/3+2\eps}}   ,& $ t<2\al/3+\eps  <\al$ 
    \end{dcases*}.
\end{align*}
The only term which is not readily in $o(1)$ is $\chi^{-\al/3+\eps+t}$. But $\al<3t/2$ in order for $2\al/3+\eps<t$ to hold so that $-\al/3+\eps + t > \eps+t/2>0$ and it follows that $\chi^{-\al/3+\eps+t}=o(1)$. This delivers~\eqref{eq:A_s_s+1_o(1)_hayman_proof} and we are done.
\end{proof}
In complete analogy to Lemma~\ref{lem:a_l_b_l_hayman_proof} we show that the functions in~\eqref{eq:tilde_a_b_ell_hayman_proof} are asymptotically equal to versions of~\eqref{eq:A_s_t_def_hayman_proof} when their argument gets close the radius of convergence $\rho$.
\begin{lemma}
\label{lem:tilde_a_l_b_l_hayman_proof}
Let $(c_k)_{k\in\Nat}\in\F (2\al/3+\eps,\al,\rho)$ for $0<\eps<\al/3, 0<\rho\le 1$ and $\al>0$. Set $r=\rho\eul^{-\chi}$ for $\chi>0$. Then the functions~\eqref{eq:def_a,b,c_set_ell_hayman_proof} fulfill for all $\ell\in\Nat_0,(p_1,\dots,p_\ell)\in\Nat_0^\ell$ that
\begin{align}
    \label{eq:tilde_a_l_asympt}
    \tilde{a}_\ell(r)
    \sim A_{1,1}(r) \quad\text{and}\quad \tilde{b}_\ell(r) \sim A_{2,2}(r), \quad \text{as }\chi\to 0.
\end{align}
\end{lemma}
\begin{proof}[Proof of Lemma~\ref{lem:tilde_a_l_b_l_hayman_proof}]
For any $p\in \Nat_0$ we have that $A_{1,2+p}(r)/(A_{1,1}(r) A_{0,1+p}(r)) = o(1)$ due to~\eqref{eq:A_s_s+1_o(1)_hayman_proof} giving us $\tilde{a}_\ell(r) \sim A_{1,1}(r)$. And with H{\"o}lder's inequality we obtain for any $p\in\Nat_0$
\begin{align}
    \label{eq:Hoelder_A_1_2_hayman_proof}
    A_{1,2+p}(r)^2
    = \bigg(\sum_{j\ge 1}\sum_{k\ge 1} j^{(p+2)/2} k\sqrt{c_kr^{jk}} \cdot j^{p/2} \sqrt{c_kr^{jk}}\bigg)^2
    \le A_{2,3+p}(r)A_{0,1+p}(r).
\end{align}
Hence~\eqref{eq:A_s_s+1_o(1)_hayman_proof} delivers that $A_{2,3+p}(r)/(A_{2,2}(r)A_{0,1+p}(r)) = o(1)$ for any $p\in\Nat_0$ giving $\tilde{b}_\ell (r) \sim A_{2,2}(r)$ as $\chi\to 0$.
\end{proof}
With these statements at hand, we are able to deliver the proof of Lemma~\ref{lem:h_admissible_gen_series}.
\begin{proof}[Proof of Lemma~\ref{lem:h_admissible_gen_series} (multiset case).]
We need to verify~\ref{item:hayman_capture}--\ref{item:hayman_decay} for~\eqref{eq:tilde_a_b_ell_hayman_proof}. Set $r=\rho\eul^{-\chi}$ for $\chi>0$. Then~\eqref{eq:tilde_a_l_asympt} shows that $a_\ell(r)\sim A_{1,1}(r)$ and $\tilde{b}_\ell(r)\sim A_{2,2}(r)$. Equations~\eqref{eq:A_s_2bigOmega_hayman_proof},\eqref{eq:A_s_t_0<rho<1_proof_hayman} for $0<\rho<1$ and~\eqref{eq:A_s_t_bigO_proof_hayman} for $\rho=1$ show that $\tilde{a}_\ell(r)$ and $\tilde{b}_\ell(r)$ both tend to infinity as $\chi\to 0$ showing~\ref{item:hayman_capture}.

As in~\eqref{eq:theta_0_def_hayman_proof} we define $\theta_0 = \chi^{1+\delta}$ for some $\al/3<\delta<\al/3+\eps/2$. The Taylor expansion of $\tilde{F}_\ell(r\eul^{\iu\theta})$ yields that for $\lvert\theta\rvert \le\theta_0$ there is some $\eta=\eta(\theta)\in (-\theta,\theta)$ such that
\[
    \tilde{F}_\ell(r\eul^{\iu\theta})
    = \tilde{F}_\ell(r) \cdot \e{\iu\theta \tilde{a}_\ell(r) - \frac{\theta^2}{2}\tilde{b}_\ell(r) + \iu^3\frac{\theta^3}{6}\tilde{c}_\ell(r\eul^{\iu\eta})}.
\]
Define
\[
    \tilde{d}_\ell(x)
    := \sum_{1\le i\le \ell} \bigg( \frac{A_{3,4+p_i}(x)}{A_{0,1+p_i}(x)}-3\frac{A_{1,2+p_i}(x)A_{2,3+p_i}(x)}{A_{0,1+p_i}(x)^2}+ 2\frac{A_{1,2+p_i}(x)^3}{A_{0,1+p_i}(x)^3}\bigg).
\]
In view of~\eqref{eq:tilde_a_b_ell_hayman_proof} we obtain for $\lvert\theta\rvert \le \theta_0$ that $\lvert\theta \tilde{c}_\ell(r\eul^{\iu\eta})\rvert \le  \theta_0^3\lvert A_{3,3}(r\eul^{\iu\eta})\rvert + \theta_0^3\lvert \tilde{d}_\ell(r\eul^{\iu\eta})\rvert$.
Since $A_{3,3}(x)$ does only have non-negative coefficients we obtain that $\lvert A_{3,3}(r\eul^{\iu\eta})\rvert \le A_{3,3}(r)$ uniformly in $\eta$ so that by~\eqref{eq:A_s_2bigOmega_hayman_proof},~\eqref{eq:A_s_t_0<rho<1_proof_hayman} (for $0<\rho<1$) and~\eqref{eq:A_s_t_bigO_proof_hayman} (for $\rho=1$) we get $\theta_0^3\lvert A_{3,3}(r\eul^{\iu\eta})\rvert \le\theta_0^3 A_{3,3}(r) =  \bigO{\chi^{3\delta - \al} } = o(1)$ as $\chi\to 0$ and uniformly in $\lvert\theta\rvert \le \theta_0$. Thus, it is left to show that $\theta_0^3\lvert \tilde{d}_\ell(r\eul^{\iu\eta})\rvert = o(1)$ to get~\ref{item:hayman_locality}. We cannot simply apply the triangle inequality to get rid of $\eul^{\iu\eta}$ as powers of $A_{0,1+p_i}$, $1\le i\le \ell$, are appearing in the denominators in the sum representing $\tilde{d}_\ell(r\eul^{\iu\eta})$. So, we first establish that $\lvert A_{0,1+p}(r\eul^{\iu\eta})\rvert \sim A_{0,1+p}(r)$ as $\chi\to0$ uniformly in $\lvert\theta\rvert\le\theta_0$ and for any $p\in\Nat_0$. 
We proceed similar to~\eqref{eq:A_0_complex_absolute_value_hayman_proof} and the subsequent text. For $j\cdot k\le \chi^{-1-\delta/2}$ we get that $\eta j k = o(1)$ since $\lvert \eta\rvert \le \theta_0=\chi^{-1-\delta}$ implying
\begin{align*}
    \lvert A_{0,1+p}(r\eul^{\iu\eta})\rvert
    \ge \bigg|\sum_{j\ge 1}j^p\sum_{k\ge 1}c_kr^{jk}\cos(\eta jk)\bigg|
    \sim  \bigg\lvert \sum_{ \substack{ j,k\ge 1,\\ jk< \chi^{-1-\delta/2}}}j^{p}c_kr^{jk} + \sum_{\substack{j,k\ge1, \\jk\ge \chi^{-1-\delta/2}}}j^{p}c_kr^{jk}
    \cos(\eta j k)\bigg\rvert .
\end{align*}
Hence, as $\chi\to0$ uniformly in $\lvert\eta\rvert\le\lvert \theta\rvert \le \theta_0 = \chi^{-1-\delta/2}$
\[
    \sum_{j\ge 1}j^p\sum_{k\ge 1}c_kr^{jk}\cos(\eta jk)
    \sim  A_{0,1+p}(r) - R_0 + R_1, 
    \quad\text{where }R_i := \sum_{j,k\ge 1, jk\ge \chi^{-1-\delta/2}}j^pc_kr^{jk}\cos(\eta j k)^i,\, i=0,1.
\]
Note that we have that $c_kr^{jk} =\bigO{k^{\al-1} \rho^{-k}\rho^{jk}\eul^{-\chi jk}}=\bigO{ k^{\al-1}\eul^{-\chi jk}}$ since $(c_k)_{k\in\Nat}\in \F(2\al/3+\eps,\al,\rho)$ and $0<\rho\le 1$. Further observe that if $1\le k\le \chi^{-1-\delta/2}$ then $\max\{1,\chi^{-1-\delta/2}/k\} = \chi^{-1-\delta/2}/k$ and $1$ otherwise. Thus we obtain for $i=0,1$
\begin{align*}
    \lvert R_i \rvert 
    &=\bigO{ \sum_{1\le k\le \chi^{-1-\delta/2}}k^{\al-1} \sum_{j\ge \chi^{-1-\delta/2}/k}j^{p}\eul^{-\chi jk}}
    + \bigO{\sum_{k> \chi^{-1-\delta/2}} k^{\al-1} \sum_{j\ge 1}  j^{p}\eul^{-\chi jk}} \\
    &= \mathcal{O}\bigg( \eul^{-\chi^{-\delta/2}}\sum_{1\le k\le \chi^{-1-\delta/2}}k^{\al-1} \frac{1}{(1- \eul^{-\chi k})^{p+1}} \bigg)
    + \mathcal{O}\bigg(\sum_{k> \chi^{-1-\delta/2}} k^{\al-1}\frac{\eul^{-\chi k}}{(1-\eul^{-\chi k})^{p+1}} \bigg).
\end{align*}
The first term is bounded by
\[
    \eul^{-\chi^{-\delta/2}}\sum_{1\le k\le \chi^{-1-\delta/2}}c_k \frac{1}{(1- \eul^{-\chi k})^{p+1}}
    =\mathcal{O}\bigg( \eul^{-\chi^{-\delta/2}} \chi^{-1-\delta/2} (\chi^{-1-\delta/2})^{\al-1} \chi^{-(p+1)} \bigg)
    =o(1).
\]
For the second term we get
\[
    \sum_{k> \chi^{-1-\delta/2}} c_k \frac{\eul^{-\chi k}}{(1-\eul^{-\chi k})^{p+1}}
    \sim \sum_{k> \chi^{-1-\delta/2}} c_k \eul^{-\chi k}
    = \mathcal{O}\bigg( \int_{ \chi^{-1-\delta/2}}^\infty x^{\al-1} \eul^{-\chi x}dx \bigg)
    =o(1).
\]
All in all, $\lvert R_i\rvert = o(1)$ for $i=0,1$. This implies that $\lvert A_{0,1+p}(r\eul^{\iu\eta})\rvert \ge A_{0,1+p}(r)+o(1)$. As additionally $A_{0,1+p}$  has only non-negative coefficients we obtain $\lvert A_{0,1+p}(r\eul^{\iu\eta})\rvert \le A_{0,1+p}(r)$. Thus, for any $p\in\Nat_0$, we conclude that $A_{0,1+p}(r\eul^{\iu\eta})\sim A_{0,1+p}(r)$ as $\chi\to 0$ and uniformly in $\lvert\eta\rvert\le \lvert\theta\rvert \le \theta_0$. Using that $A_{s,t}$ has only non-negative coefficients for any $s,t\in\Nat_0$ we then get the estimate
\begin{align*}
    \lvert \tilde{d}_\ell(r\eul^{\iu\eta}) \rvert 
    &\le \sum_{1\le i\le \ell}  \left(\frac{A_{3,4+p_i}(r)}{\lvert A_{0,1+p_i}(r\eul^{\iu\eta})\rvert} + 3\frac{A_{1,2+p_i}(r)A_{2,3+p_i}(r)}{\lvert  A_{0,1+p_i}(r\eul^{\iu\eta})\rvert^2} + 2\frac{A_{1,2+p_i}(r)^3}{\lvert A_{0,1+p_i}(r\eul^{\iu\eta})\rvert^3}\right) \\
    &\sim \sum_{1\le i\le \ell}\left(\frac{A_{3,4+p_i}(r)}{A_{0,1+p_i}(r)} + 3\frac{A_{1,2+p_i}(r)A_{2,3+p_i}(r)}{A_{0,1+p_i}(r)^2} + 2\frac{A_{1,2+p_i}(r)^3}{A_{0,1+p_i}(r)^3}\right).
\end{align*}
With H{\"o}lder's inequality we obtain for any $p\in\Nat_0$
\begin{align*}
     A_{2,3+p}(r)^2 
     &= \bigg( \sum_{j\ge 1}\sum_{k\ge 1}j^{p/2+3/2}k^{3/2}\sqrt{c_k r^{jk}} \cdot j^{p/2+1/2}\sqrt{kc_kr^{jk}} \bigg)^2
    \le A_{3,4+p}(r)A_{1,2+p}(r) \quad\text{and}  \\
    A_{1,2+p}(r)^3 
    &= \bigg(\sum_{k\ge 1} j^{p/3+1} k (c_kr^{jk})^{1/3} \cdot j^{2p/3}(c_kr^{jk})^{2/3} \bigg)^3
    \le A_{3,4+p}(r) A_{0,1+p}(r)^2.
\end{align*}
The bounds in the previous display together with $A_{1,2+p}(r)^2 \le A_{2,3+p}(r)A_{0,1+p}(r)$, which was showed in~\eqref{eq:Hoelder_A_1_2_hayman_proof}, entail the estimates 
\[
    \frac{A_{1,2+p}(r)A_{2,3+p}(r)}{A_{0,1+p}(r)^2} \le \frac{A_{2,3+p}(r)^2}{A_{0,1+p}(r)A_{1,2+p}(r)} 
    \le \frac{A_{3,4+p}(r)}{A_{0,1+p}(r)}\quad\text{and}\quad
    \frac{A_{1,2+p}(r)^3}{A_{0,1+p}(r)^3} \le \frac{A_{3,4+p}(r)}{A_{0,1+p}(r)}
\] 
hold true. We conclude that
\[
    \lvert\tilde{d}_\ell(r\eul^{\iu\eta})\rvert 
    = \bigO{\sum_{1\le i\le \ell}\frac{A_{3,4+p_i}(r)}{A_{0,1+p_i}(r)}}.
\]
Since according to~\eqref{eq:A_s_s+1_o(1)_hayman_proof} we have for any $p\in\Nat_0$ that $A_{3,4+p}(r)/(A_{0,1+p}(r)A_{3,3}(r)) = o(1)$ as $\chi\to 0$ we obtain in turn $\theta_0^3\lvert  \tilde{d}_\ell(r\eul^{\iu\eta})\rvert = o(\theta_0^3A_{3,3}(r))=o(1)$ as $\chi\to 0$ and uniformly in $\lvert\eta\rvert\le\lvert\theta\rvert\le\theta_0$ finishing the proof for~\ref{item:hayman_locality}.

In order to show~\ref{item:hayman_decay} we use that $\sum_{j\ge 1}j^{p_i}C(x^j)$ has only non-negative coefficients and obtain similar to~\eqref{eq:hayman_decay_proof_1} 
\[
    \bigg\lvert \frac{\tilde{F}_\ell(r\eul^{\iu \theta})}{\tilde{F}_\ell(r)} \bigg\rvert
    \le \bigg\lvert \frac{G(r\eul^{\iu \theta})}{G(r)} \bigg\rvert
    = \e{-2\sum_{j\ge 1} j^{-1} \sum_{k\ge 1}c_kr^{jk}\sin^2(\theta j k/2) }
    \le \e{-2\sum_{k\ge 1}c_kr^k\sin^2(\theta k/2)}.
\]
Noting that we chose the same $\theta_0$ as in~\eqref{eq:theta_0_def_hayman_proof} we obtain as in~\eqref{eq:hayman_decay_proof_2} that there is some $f>0$ yielding as $\chi\to 0$ and uniformly in $ \theta_0 \le \lvert\theta\rvert <\pi$
\[
    \bigg\lvert \frac{\tilde{F}_\ell(r\eul^{\iu \theta})}{\tilde{F}_\ell(r)} \bigg\rvert
    \le \eul^{- \chi^{-f}}.
\]
Since $\tilde{b}_\ell(r) \sim A_{2,2}(r)=\Omega(\chi^{-2\al/3+\eps+2})$ due to~\eqref{eq:tilde_a_l_asympt} as well as~\eqref{eq:A_s_2bigOmega_hayman_proof},~\eqref{eq:A_s_t_0<rho<1_proof_hayman} (for $0<\rho<1$) and~\eqref{eq:A_s_t_bigO_proof_hayman} (for $\rho=1$) this shows~\ref{item:hayman_decay} and we are done.
\end{proof}

\subsection{Proof of the Main Results}
\label{sec:proof_main}
\subsubsection{Proofs for Coefficient Extraction and Counting} 
\label{sec:coro:coeff_eC(x)_C(x)^k}
\begin{proof}[Proof of Theorem~\ref{thm:coeff_eC(x)_C(x)^k} (set case).]
Let $z_n$ be such that $z_nC'(z_n)=n$. By applying Lemma~\ref{lem:asymptotics_expansive_sum} we obtain that as $\eta \to 0$,
\begin{equation}
\label{eq:CC'C''}
    C(\rho e^{-\eta}) = \Theta(h(\eta^{-1}) \eta^{-\alpha}),
    \quad
    C'(\rho e^{-\eta}) = \Theta(h(\eta^{-1})\eta^{-(\alpha+1)}),
    \quad
    C''(\rho e^{-\eta}) = \Theta(h(\eta^{-1})\eta^{-(\alpha+2)}).
\end{equation}
In particular, if we set $z_n=\rho \eul^{-\eta_n}$, then $\eta_n \to 0$ as $n\to\infty$. With this observation let us proceed with extracting the $n$-th coefficient in $F_\ell(x) = S(x)C(x)^\ell$.
Lemma~\ref{lem:h_admissible_gen_series} guarantees that $F_\ell$ is $H$-admissible for any $\ell\in\Nat_0$, and so, by applying Lemma~\ref{lem:hayman} we obtain that
\[
    [x^n]S(x)C(x)^\ell 
    = \frac{S(z_n)C(z_n)^\ell}{\sqrt{2\pi z_n^2C''(z_n)}} z_n^{-n} \bigg( \e{-\frac{(a_\ell(z_n)-n)^2}{2b_\ell(z_n)}} + o(1)\bigg),
\]
where $a_\ell(x)$ and $b_\ell(x)$ are the functions defined in~\eqref{eq:a,b,c_hayman_def} for $f_\ell(x) = \log F_\ell(x) = C(x) + \ell \ln C(x)$, see~\eqref{eq:def_a_ell}. Moreover, from~\eqref{eq:def_a,b,c_set_ell_hayman_proof} and by plugging in~\eqref{eq:CC'C''} we obtain that
\[
    a_\ell(z_n) - n = A_1(z_n) + \ell \frac{A_1(z_n)}{A_0(z_n)}  - n = \Theta(\eta_n^{-1}),
    \qquad
    b_\ell(z_n) = \Theta(h(\eta_n^{-1}) \eta_n^{-(\alpha+2)}).
\]
Since $\alpha > 0$, the claim follows.
\end{proof}
\begin{proof}[Proof of Theorem~\ref{thm:coeff_eC(x)_C(x)^k} (multiset case).]
In the forthcoming argumentation recall the definitions of $A_s(x)$ and $A_{s,t}(x)$ from~\eqref{eq:A_s_def_hayman_proof} and~\eqref{eq:A_s_t_def_hayman_proof}. In particular, note that
\[
    A_0(\rho\eul^{-\eta})=C(\rho\eul^{-\eta}),\quad
     A_1(\rho\eul^{-\eta})= \rho\eul^{-\eta}C'(\rho\eul^{-\eta})
    \quad\text{and}\quad 
     A_2(\rho\eul^{-\eta})\sim (\rho\eul^{-\eta})^2 C''(\rho\eul^{-\eta}),\quad \text{as }\eta\to 0
\]
as well as
\begin{align*}
    A_{0,0}(\rho\eul^{-\eta})
    &= \sum_{j\ge1}C(\rho\eul^{-\eta})/j , \quad 
    A_{1,1}(\rho\eul^{-\eta})
    =\sum_{j\ge 1} (\rho\eul^{-\eta})^jC'((\rho\eul^{-\eta})^j) 
    \quad\text{and} \\
    A_{2,2}(\rho\eul^{-\eta})
    &\sim \sum_{j\ge 1}j(\rho\eul^{-\eta})^{2j}C''((\rho\eul^{-\eta})^j)
    ,\quad \text{as }\eta\to0.
\end{align*}
Let $z_n = \rho\eul^{-\eta_n}$ be the solution to $z_nC'(z_n)=n$ and $q_n=\rho\eul^{-\xi_n}$ to $\sum_{j\ge 1}q_n^jC'(q_n^j) = n$.
Due to~\eqref{eq:A_s_t_bigO_proof_hayman} and~\eqref{eq:CC'C''} we have that $\eta_n\to0 $ as $n\to\infty$ for $0<\rho\le 1$ and $\xi_n\to0$ as $n\to\infty$ for $0<\rho<1$. Equation~\eqref{eq:A_s_t_0<rho<1_proof_hayman} implies that $\xi_n\to 0$ as $n\to\infty$. With this at hand, we proceed extracting the $n$-th coefficient in $\tilde{F}_\ell(x) = G(x)\prod_{1\le i\le \ell}\sum_{j\ge 1}j^{p_i}C(x^j)$. Due to Lemma~\ref{lem:h_admissible_gen_series} we know that $\tilde{F}_\ell$ is $H$-admissible so that applying Lemma~\ref{lem:hayman} delivers 
\[
    [x^n]\tilde{F}_\ell(x)
    = \frac{\tilde{F}_\ell(w_n)}{\sqrt{2\pi \tilde{b}_\ell(\omega_n)}} 
    w_n^{-n} \bigg( \e{-\frac{(\tilde{a}_\ell(\omega_n)-n)^2}{2\tilde{b}_\ell(\omega_n}) } + o(1) \bigg), \qquad w_n\in\{z_n,q_n\},
\]
where $\tilde{a}_\ell$ and $\tilde{b}_\ell$ are the functions defined in~\eqref{eq:a,b,c_hayman_def} for $\tilde{f}_\ell = \log \tilde{F}_\ell$, see~\eqref{eq:tilde_a_b_ell_hayman_proof}. Equation~\eqref{eq:tilde_a_l_asympt} yields $\tilde{b}_\ell(w_n) \sim A_{2,2}(w_n)$ for $w_n\in\{z_n,q_n\}$. Recalling~\eqref{eq:tilde_a_b_ell_hayman_proof} and using \eqref{eq:A_s_t_0<rho<1_proof_hayman} for $0<\rho<1$ we further obtain for $0<\rho\le 1$ that $\tilde{a}_\ell(w_n) = n + \mathcal{O}\big( \sum_{1\le i\le \ell}A_{1,2+p_i}(w_n)/A_{0,1+p_i}(w_n)\big)$. We deduce
\[
    [x^n]\tilde{F}_\ell(x)
    = \frac{\tilde{F}_\ell(w_n)}{\sqrt{2\pi A_{2,2}(w_n)}} 
    w_n^{-n} \bigg( \e{\bigO{\bigg(\sum_{1\le i\le \ell}\frac{A_{1,2+p_i}(w_n)}{A_{0,1+p_i}(w_n)}\bigg)^2\frac{1}{A_{2,2}(w_n)}} } + o(1) \bigg), \qquad w_n\in\{z_n,q_n\}.
\]
According to~\eqref{eq:A_s_2bigOmega_hayman_proof},~\eqref{eq:A_s_t_0<rho<1_proof_hayman} for $0<\rho<1$ and~\eqref{eq:A_s_t_bigO_proof_hayman} for $\rho=1$  we obtain for any $p\in\Nat_0$ that
\[
    \frac{A_{1,2+p}(w_n)^2}{A_{0,1+p}(w_n)^2A_{2,2}(w_n)}
    = \bigO{\beta_n^{3\eps}} 
    = o(1), \quad \beta_n\in\{\eta_n,\xi_n\}.
\]
Hence $[x^n]\tilde{F}_\ell(x) \sim \tilde{F}_\ell(w_n)/\sqrt{2\pi A_{2,2}(w_n)}w_n^{-n}$ for $w_n\in\{z_n,q_n\}$.
Note that we can rewrite $\tilde{F}_\ell(x) = \e{A_{0,0}(x)}\prod_{1\le i\le \ell}A_{0,p_i}(x)$. Accordingly, if $0<\rho<1$, we get due to Equation~\eqref{eq:A_s_t_0<rho<1_proof_hayman} that $\tilde{F}_\ell(z_n) \sim \e{\sum_{j\ge 2}C(\rho^j)/j} \e{C(z_n)} C(z_n)^\ell$.
In addition, if $\rho<1$,~\eqref{eq:A_s_t_0<rho<1_proof_hayman} gives that $A_{2,2}(z_n)\sim z_n^2C''(z_n)$. So,~\eqref{eq:coeff_ogf_muliset_rho<1} follows. If $\rho=1$ we use that
$A_{2,2}(q_n) \sim \sum_{j\ge 1}jq_n^{2j}C''(q_n^j)$. This entails~\eqref{eq:coeff_ogf_muliset_rho=1} and we are done.
\end{proof}

\subsubsection{Proofs for the Distribution of the Largest and Smallest clusters}
\label{sec:proof_smallest/largest}

\subsubsection*{Proof of Theorem~\ref{thm:largest_comp_hayman}} 
\label{sec:thm:largest_comp_hayman}

\begin{proof}[Proof of Theorem~\ref{thm:largest_comp_hayman}.]
We conduct the proof for $\mathcal{L}(\mathsf{S}^{(n)})$ and $\mathcal{L}(\mathsf{G}^{(n)})$ simultaneously. Therefore let $\mathsf{F}^{(n)}\in\Omega_n$ be either $\mathsf{S}^{(n)}$ or $\mathsf{G}^{(n)}$ and we will specify when there is need to differentiate the two cases. We start with some statements about $z_n=\rho\eul^{-\eta_n}$ and $q_n=\rho\eul^{-\xi_n}$ solving $z_nC'(z_n)=n$ and $\sum_{j\ge 1}q_n^jC'(q_n^j)=n$, respectively. Set 
\begin{align}
	\label{eq:def_w_beta_hayman}
	w_n
	=\rho\eul^{-\beta_n} 
	= \begin{dcases*}
		z_n, &$\mathsf{F}^{(n)}=\mathsf{S}^{(n)}$ or $\mathsf{F}^{(n)}=\mathsf{G}^{(n)}$ and $0<\rho<1$ \\
		q_n, &$\mathsf{F}^{(n)}=\mathsf{G}^{(n)}$ and $\rho=1$
	\end{dcases*}
	\quad\text{and}\quad 
	\beta_n =\ln(\rho/w_n).
\end{align}
According to~\eqref{eq:sv_estimate_hayman} we obtain the bounds $k^{-\delta}\le h(k)\le k^{\delta}$ for  any $0<\delta<\al$ and $k$ sufficiently large; so $(c_k)_{k\in\Nat}\in \F(\al-\delta,\al+\delta,\rho)$. Hence, due to~\eqref{eq:A_s_2bigOmega_hayman_proof},~\eqref{eq:A_s_t_0<rho<1_proof_hayman} and~\eqref{eq:A_s_t_bigO_proof_hayman}
$$
    n = 
    z_nC'(z_n) =
    \mathcal{O}(\eta_n^{-(\al-\delta+1)})\cap\Omega(\eta_n^{-(\al+\delta+1)}),
    \quad
    n=\sum_{j\ge 1}q_n^jC'(q_n^j) = \mathcal{O}(\xi_n^{-(\al-\delta+1)})\cap\Omega(\xi_n^{-(\al+\delta+1)}).
$$
Thus 
\begin{align}
	\label{eq:set_max_xi_n_hayman}
	\beta_n = \mathcal{O}(n^{-1/(\al+\delta+1)}) \cap \Omega(n^{-1/(\al-\delta+1)})
	\quad\text{and}\quad 
	C(w_n) 
	= \mathcal{O}(n^{(\al+\delta)/(\al-\delta+1)}) \cap \Omega(n^{(\al-\delta)/(\al+\delta+1)}).
\end{align}
We proceed with the proof of the theorem. We obtain that 
\begin{align}
	\label{eq:set_max_prob_as_set_hayman}
	\pr{\mathcal{L}(\mathsf{F}^{(n)})\le s}
	=\pr{\mathsf{F}^{(n)}\in \Omega_{n,\le s}},
	\quad	\Omega_{n,\le s}:=\{(N_1,\dots,N_n)\in\Omega_n: \forall s<i\le n: N_i=0\}.
\end{align}
Further, for $\ell \in \Nat$ and $\mathrm{k}=(k_1,\dots,k_\ell)\in\Nat^\ell$ let 
\begin{align*}
	\Omega_{n,\mathrm{k}}
	:= \{(N_1,\dots,N_n)\in\Omega_n:\forall 1\le i\le \ell: N_{k_i}\ge 1\}.
\end{align*}
With this at hand, we obtain
\begin{align}
	\label{eq:largest_comp_proof_Omega_les}
	\Omega_{n,\le s}
	= \Omega_{n} \setminus \bigcup_{k>s}\Omega_{n,k}.
\end{align}
Define
\[
	I_\ell
	=I_\ell(n) 
	:= \sum_{s<k_1<\cdots< k_\ell \atop k_1+\cdots+k_\ell \le n}\pr{\mathsf{F}^{(n)} \in \Omega_{n,\mathrm{k}}},
	\qquad \ell\in\Nat.
\]
Here comes the inclusion/exclusion principle into play. In light of~\eqref{eq:set_max_prob_as_set_hayman} and~\eqref{eq:largest_comp_proof_Omega_les} we obtain that
\begin{align*}
	\pr{\mathcal{L}(\mathsf{F}^{(n)})\le s}
	&= \pr{\mathsf{F}^{(n)}\in \Omega_n\setminus \bigcup_{k>s}\Omega_{n,k}}
	= \pr{\mathsf{F}^{(n)}\in \Omega_n}
	+\sum_{\ell\ge 1}(-1)^\ell I_\ell  
	= 1 + \sum_{\ell\ge 1}(-1)^\ell I_\ell .
\end{align*}
Computing the union of events by the inclusion/exclusion principle entails the helpful ``sandwich-property'' that for any $M>1$
\begin{align}
	\label{eq:set_max_inclusion_sandwich_hayman}
	1+\sum_{1\le \ell\le 2M-1}(-1)^\ell I_\ell
	\le \pr{\mathcal{L}(\mathsf{F}^{(n)})\le s}
	\le 1+\sum_{1\le \ell\le 2M}(-1)^\ell I_\ell.
\end{align}
This has the great advantage that we can take a large but fixed $M$ when analysing $\mathcal{L}(\mathsf{F}^{(n)})$;  investigating $I_\ell=I_\ell(n)$ for fixed $\ell$ is much easier, as we only need to let one parameter (namely $n$) tend to infinity. 
Recall that for $t\in\Real$ we defined
\[
	s_n
	= s(t,\beta_n)
	:= \beta_n^{-1} \big( \ln X(\beta_n) + t\big) ,\quad X(\beta_n)
	:=\frac{1}{\Gamma(\al)} C(w_n) (\ln C(w_n))^{\al-1}
	\frac{h(\beta_n^{-1}\ln C(w_n))}{h(\beta_n^{-1})}.
\]
Since $h(\beta_n^{-1}\ln C(w_n))/h(\beta_n^{-1})=\bigO{\ln C(w_n)}$ due to~\eqref{eq:sv_estimate_hayman} we obtain 
\begin{align}
	\label{eq:s(t,beta)_asymptotic_hayman}
	s_n
	\sim \beta_n^{-1} \ln C(w_n), \quad n\to\infty.
\end{align}
In order to get a grip on~\eqref{eq:set_max_inclusion_sandwich_hayman} we claim that
\begin{align}
	\label{eq:F_n_k_i_as_F_and_c_i_sim_hayman}
	\pr{\mathsf{F}^{(n)}\in \Omega_{n,\mathrm{k}}}
	\sim \prod_{1\le i\le \ell}c_{k_i}w_n^{k_i}, \quad 
	\ell\in\Nat,\, s_n< k_1< \cdots < k_\ell < s_n + o(s_n)
\end{align}
and for any $\eps>0$ when $n$ sufficiently large
\begin{align}
	\label{eq:F_n_k_i_as_F_and_c_i_le_hayman}
	\pr{\mathsf{F}^{(n)}\in \Omega_{n,\mathrm{k}}}
	\le (1+\eps)  \cdot \prod_{1\le i\le \ell}c_{k_i}w_n^{k_i}, \quad 
	\ell\in\Nat,\, k_1,\dots,k_\ell > 0.
\end{align}
The proof of~\eqref{eq:F_n_k_i_as_F_and_c_i_sim_hayman} and~\eqref{eq:F_n_k_i_as_F_and_c_i_le_hayman}, which is rather lengthy and relies heavily upon the underlying generating series to be $H$-admissible, is deferred to the end of this section for better readability.  
We continue by showing
\begin{align}
	\label{eq:set_max_A_ell_hayman}
	I_\ell 
	\sim \frac{1}{\ell!}(\eul^{-t})^\ell, \quad  \ell \in \Nat.
\end{align}
Then, by choosing $M\in\Nat$ sufficiently large in~\eqref{eq:set_max_inclusion_sandwich_hayman}, we get that $\pr{\mathcal{L}(\mathsf{F}^{(n)}) \le s_n}$ gets arbitrarily close to $1 + \sum_{\ell\ge 1}(-1)^\ell (\eul^{-t})^\ell/\ell! = \eul^{-\eul^{-t}}$ and the claim of Theorem~\ref{thm:largest_comp_hayman} follows. 
We first treat the case $\ell=1$ in~\eqref{eq:set_max_A_ell_hayman}, which is instructive for the remaining cases $\ell\ge 2$.
Let $\nu\equiv \nu(\beta_n)$ be in $\omega(\beta_n^{-1})\cap o(s_n)$, which is possible due to~\eqref{eq:s(t,beta)_asymptotic_hayman}. Note further that $\nu=o(\beta_n^{-1}\ln C(w_n))=o(n)$ as $n\to\infty$ according to~\eqref{eq:set_max_xi_n_hayman}. Then 
\[
	I_1
	= \bigg( \sum_{s_n<k\le s_n+\nu} + \sum_{k>s_n+\nu}\bigg) \pr{\mathsf{F}^{(n)}\in \Omega_{n,k}}
	=: I_{1,1} + I_{1,2}.
\]
We are first going to show that $I_{1,1}\sim\eul^{-t}$. Due to~\eqref{eq:F_n_k_i_as_F_and_c_i_sim_hayman} we obtain
\begin{align}
	\label{eq:set_max_A_1_1_computation_hayman}
	I_{1,1} 
	\sim \sum_{s_n<k\le s_n+\nu} c_kw_n^k
	= \sum_{s_n<k\le s_n+\nu} h(k) k^{\al-1} \eul^{-\beta_n k}
	\sim h(s_n) s_n^{\al-1} \eul^{-\beta_n s_n} \sum_{0<k \le \nu} \eul^{-\beta_n k}
	\sim h(s_n) s_n^{\al-1} \eul^{-\beta_n s_n}  \beta_n^{-1}.
\end{align}
Recalling $s_n\sim \beta_n^{-1}\ln C(w_n)$ from~\eqref{eq:s(t,beta)_asymptotic_hayman} and plugging in $s_n$ into the expression in the previous display gives by Lemma~\ref{lem:asymptotics_expansive_sum}
\begin{align}
	\label{eq:set_max_A_1_1_hayman}
	I_{1,1}
	\sim \frac{\Gamma(\al) h(\beta_n^{-1})\beta_n^{-\al}}{C(w_n)}\cdot \eul^{-t}
	\sim \eul^{-t}.
\end{align}
For $k>s_n+\nu$ we use the estimate~\eqref{eq:F_n_k_i_as_F_and_c_i_le_hayman} and obtain
\begin{align*}
	\pr{\mathsf{F}^{(n)}\in \Omega_{n,k}}
	= \bigO{c_kw_n^k}.
\end{align*}
With this at hand and since $\nu=\omega(\beta_n^{-1}$) Lemma~\ref{lem:asymptotics_expansive_integral_arb_start} reveals for $k>s_n+\nu$ that
\begin{align}
	\label{eq:set_max_A_1_2_hayman}
	I_{1,2}
	= \bigO{\sum_{k>s_n+\nu}c_kw_n^k}
	= \bigO{h(s_n)s_n^{\al-1}\eul^{-\beta_n (s_n+\nu)}\beta_n^{-1}}
	= \bigO{\eul^{-t} \eul^{-\beta_n \nu}} = o(1).
\end{align}  
It follows that $I_1\sim\eul^{-t}$, as claimed. Next, we treat the remaining $\ell\ge 2$ in~\eqref{eq:set_max_A_ell_hayman}. Define
$$
    {B}
	:=
	\{\mathrm{k}\in \Nat^\ell:\forall 1\le i\le \ell: k_i>s_n, k_1+\cdots+k_\ell\le n\},
$$
and
$$
	{B}_=
	:= \{\mathrm{k}\in{B}: \exists 1\le i< j \le \ell: k_i=k_j\}.
$$
In addition we need the sets
\begin{align*}
	{B}_\le 
	:= \{\mathrm{k}\in{B}: \forall 1\le i\le \ell: k_i\le s_n+\nu\},
	\quad
	{B}_>
	:= \{\mathrm{k}\in{B}: \exists 1\le i\le \ell: k_i> s_n+\nu\}.
\end{align*}
Further set 
\[
	f(\mathrm{k})
	:= \prod_{1\le i\le k}c_{k_i}w_n^{k_i} ,
	\quad\mathrm{k}\in{B}.
\]
With these definitions at hand and since $\nu = o(s_n)$ we obtain by applying ~\eqref{eq:F_n_k_i_as_F_and_c_i_sim_hayman} and~\eqref{eq:F_n_k_i_as_F_and_c_i_le_hayman}
\begin{align*}
	\nonumber
	I_\ell 
	&= \bigg( \sum_{\mathrm{k}\in {B}_\le ,\atop k_1<\cdots<k_\ell}
	+\sum_{\mathrm{k}\in{B}_>, \atop k_1<\cdots<k_\ell } \bigg)
	\pr{\mathsf{F}^{(n)}\in \Omega_{n,k_1,\dots,k_\ell}} 
	\sim  \frac{1}{\ell!}\bigg(\sum_{{B}_\le}-\sum_{{B}_=}\bigg)f(\mathrm{k}) 
	+ \mathcal{O}\bigg(\sum_{{B}_>}f(\mathrm{k})\bigg).
\end{align*}
where we slightly abuse the notation and write $\sum_{\mathcal{X}}=\sum_{x\in \mathcal{X}}$ for any set $\mathcal{X}$.
We will prove
\begin{align}
	\label{eq:set_max_order_of_sums_hayman}
	\frac{1}{\ell!}\sum_{{B}_\le}f(\mathrm{k})
	\sim \frac{1}{\ell!}(\eul^{-t})^\ell,\quad 
	\sum_{{B}_>}f(\mathrm{k}) = o(1) \quad\text{and}\quad 
	\sum_{{B}_=}f(\mathrm{k})=o(1),
\end{align}
from which~\eqref{eq:set_max_A_ell_hayman} follows immediately.
We start with the first asymptotic identity in~\eqref{eq:set_max_order_of_sums_hayman}. 
The estimate $k_i\le s_n+ \nu$ for all $1\le i\le \ell$ implies that $k_1+\cdots+k_\ell \le \ell(s_n+\nu)=o(n)$  so that with~\eqref{eq:set_max_A_1_1_hayman}
\[
	\sum_{{B}_\le}f(\mathrm{k}) 
	= \sum_{{B}_\le} \prod_{1\le i\le \ell}c_{k_i}w_n^{k_i} 
	= I_{1,1}^\ell 
	\sim (\eul^{-t})^\ell.
\]
Next we show the second claim in~\eqref{eq:set_max_order_of_sums_hayman}. By applying~\eqref{eq:set_max_A_1_2_hayman} we obtain
\[
	\sum_{{B}_>}f(\mathrm{k})
	= \sum_{{B}_>} \prod_{1\le i\le \ell}c_{k_i}w_n^{k_i}
	\le I_{1,2}\cdot I_{1}^{\ell-1}
	= o(\eul^{-t(\ell-1)})
	= o(1).
\]
Finally, we show the third asymptotic identity in~\eqref{eq:set_max_order_of_sums_hayman}. 
We get
\[
	\sum_{{B}_{=}}f(\mathrm{k})
	\le \binom{\ell}{2} \sum_{s_n<k_1,\dots,k_{\ell-1}\le n}f(k_1,k_1,k_2,\dots,k_\ell)
	\le \binom{\ell}{2} \sum_{s_n<k\le n}c_k^2 w_n^{2k} \cdot I_{1}^{\ell-2}.
\]
According to Lemma~\ref{lem:asymptotics_expansive_integral_arb_start} and analogous to~\eqref{eq:set_max_A_1_2_hayman} we find that
\[
	\sum_{k>s_n+\nu}c_k^2 w_n^{2k}
	= \bigO{h(s_n)^2s_n^{2(\al-1)}\eul^{-2\beta_n(s_n+\nu)} \beta_n^{-1}}
	=o(1).
\]
Additionally, just as in~\eqref{eq:set_max_A_1_1_computation_hayman} and~\eqref{eq:set_max_A_1_1_hayman} we compute that 
\[
	\sum_{s_n<k\le s_n+\nu}c_k^2 w_n^{2k}
	\sim h(s_n)^2 s_n^{2(\al-1)}\eul^{-2\beta_n s_n} \sum_{1\le k\le \nu} \eul^{-2\beta_n k}
	=\bigO{ (h(s_n)s_n^{\al-1}\eul^{-\beta_n s_n} \beta_n^{-1})^2 \beta_n}
	= \bigO{\eul^{-2t} \beta_n}=o(1).
\]
Since $I_{1}^{\ell-2} \sim \eul^{-t(\ell-2)}$ this shows that $\sum_{{B}_{=}}f(\mathrm{k})=o(1)$ and we haven proven~\eqref{eq:set_max_order_of_sums_hayman} and we are done.

\vspace{3mm}
\noindent
\textit{Proof of~\eqref{eq:F_n_k_i_as_F_and_c_i_sim_hayman} and~\eqref{eq:F_n_k_i_as_F_and_c_i_le_hayman}.} 
Let $\ell\in\Nat$, $\mathrm{k}:=(k_1,\dots,k_\ell)\in\Nat^\ell$ and recall that
\[
	\Omega_{n,\mathrm{k}}
	:= \{(N_1,\dots,N_n)\in \Omega_n: \forall 1\le i\le\ell: N_{k_i}\ge 1\}.
\]
We further abbreviate
\[
	\sum_{\Omega_{n,\mathrm{k}}}  = \sum_{(N_1,\dots,N_n)\in\Omega_{n,\mathrm{k}}}, \qquad 
	\Sigma_{\mathrm{k}} := k_1+\cdots+k_\ell
	\qquad\text{and}\qquad 
	[n]_{\neq\mathrm{k}} 
	:= \{1,\dots,n\} \setminus\{k_1,\dots,k_\ell\}.
\]
From here on we treat the two cases ``set'' and ``multiset'' seperately. 
\paragraph{The set case.} 
From~\eqref{eq:distribution_cluster_set} we obtain
\begin{align}  \label{eq:set_cluster_as_sum_hayman}
	\pr{\mathsf{S}^{(n)} \in \Omega_{n,\mathrm{k}}}
	&= \frac{1}{[x^n]S(x)} \sum_{\Omega_{n,\mathrm{k}}}	
	\prod_{1\le i\le n} \frac{c_i^{N_i}}{N_i!}
	= \frac{1}{[x^n]S(x)} \cdot \prod_{1\le i\le \ell} c_{k_i} \cdot 
	\sum_{\Omega_{n,\mathrm{k}}}	
	\prod_{i\in[n]_{\neq\mathrm{k}} } \frac{c_i^{N_i}}{N_i!}
	\cdot \prod_{1\le i\le \ell} \frac{c_{k_i}^{N_{k_i}-1}}{N_{k_i}!}.
\end{align}
First observe that if $(N_1,\dots,N_n)\in\Omega_{n,\mathrm{k}}$ we necessarily have that $N_i=0$ for $i> n-\Sigma_{\mathrm{k}}$.  Since further for $N_{k_i}\ge 1$ the estimate $c_{k_i}^{N_{k_i}-1}/N_{k_i}!\le c_{k_i}^{N_{k_i}-1}/(N_{k_i}-1)!$ holds for $1\le i \le \ell$ we get that 
\begin{align*}
	\sum_{\Omega_{n,\mathrm{k}}}	
	\prod_{i\in[n]_{\neq\mathrm{k}} } \frac{c_i^{N_i}}{N_i!}
	\cdot \prod_{1\le i\le \ell} \frac{c_{k_i}^{N_{k_i}-1}}{N_i!}
	&\le \sum_{\Omega_{n-\Sigma_{\mathrm{k}}}} 
	\prod_{i=1}^{n-\Sigma_{\mathrm{k}}} \frac{c_i^{N_i}}{N_i!} 
	= [x^{n-\Sigma_{\mathrm{k}}}]S(x).
\end{align*}
All in all, we have shown that
\begin{align}
	\label{eq:set_cardinalities_as_coefficients_hayman}
	 \pr{\mathsf{F}^{(n)} \in \Omega_{n,\mathrm{k}}}
	\le \frac{[x^{n-\Sigma_{\mathrm{k}}}]S(x)}{[x^n]S(x)} \cdot \prod_{1\le i\le k}c_{k_i}.
\end{align}
Lemma~\ref{lem:h_admissible_gen_series} asserts that $S(x)$ is $H$-admissible. Recall the definition of $A_s(x)$ from~\eqref{eq:A_s_def_hayman_proof} and note that the functions $a(x),b(x)$ from~\eqref{eq:a,b,c_hayman_def} for $f(x)=\ln S(x)$ are exactly $A_1(x),A_2(x)$, compare to~\eqref{eq:def_a,b,c_set_ell_hayman_proof} with $\ell=0$. We chose $z_n$ such that $A_1(z_n)=z_nC'(z_n)=n$. Let $\eps>0$. Then Lemma~\ref{lem:hayman} reveals that there is some (potentially large but fixed) $K=K(\eps)>0$ such that for $n$ sufficiently large
\begin{align}
	\label{eq:set_cluster_as_sum_upper_bound_hayman}
	\frac{[x^{n-\Sigma_{\mathrm{k}}}]S(x)}{[x^n]S(x)}
	\le (1+\eps) \cdot z_n^{\Sigma_{\mathrm{k}}},
	\quad \text{uniformly in }\mathrm{k}\in\Nat^\ell\text{ with }0<\Sigma_{\mathrm{k}}\le n-K.
\end{align}
For all $\mathrm{k}\in\Nat^\ell$ such that $n-K< \Sigma_{\mathrm{k}} \le n$ we have that $[x^{n-\Sigma_{\mathrm{k}}}]S(x)=\bigO{1}$ and hence with Lemma~\ref{lem:hayman}, noting that $z_n\le\rho\le1$,
\begin{align}
	\label{eq:set_cluster_as_sum_upper_bound_2_hayman}
	\frac{[x^{n-\Sigma_{\mathrm{k}}}]S(x)}{[x^n]S(x)}
	=\bigO{ z_n^{n}\sqrt{b(z_n)}/S(z_n) }
	= o\big(z_n^{\Sigma_{\mathrm{k}}}\big)
	, \quad\text{uniformly in }n-K< \Sigma_{\mathrm{k}} \le n.
\end{align}
In any case, by combining~\eqref{eq:set_cardinalities_as_coefficients_hayman}--\eqref{eq:set_cluster_as_sum_upper_bound_2_hayman}, and since $\ell$ is fixed, we obtain for $n$ sufficiently large
\begin{align}
	\label{eq:set_cluster_as_sum_upper_bound_final_hayman}
	\pr{\mathsf{S}^{(n)} \in \Omega_{n,\mathrm{k}}}
	\le (1+\eps) \cdot \prod_{1\le i\le \ell}c_{k_i}z_n^{k_i}, \quad \text{uniformly in }\mathrm{k}\in\Nat^\ell\text{ with }
	0< \Sigma_{\mathrm{k}} \le n.
\end{align}
This shows~\eqref{eq:F_n_k_i_as_F_and_c_i_le_hayman}. Let us next demonstrate that~\eqref{eq:F_n_k_i_as_F_and_c_i_sim_hayman} is valid. In light of~\eqref{eq:set_cluster_as_sum_upper_bound_final_hayman} it is left to show that 
\begin{align}
	\label{eq:set_cluster_as_sum_lower_bound_final_hayman}
	\pr{\mathsf{S}^{(n)} \in \Omega_{n,\mathrm{k}}}
	\ge (1+o(1)) \cdot \prod_{1\le i\le \ell}c_{k_i}z_n^{k_i}
	\quad \text{for }s_n <k_1<\cdots <k_\ell < s_n + o(s_n) \text{ and as }n\to\infty.
\end{align}
For that let $S_{\neq\mathrm{k}}(x)$ be the generating series of elements such that there are no clusters of sizes $k_1,\dots,k_\ell$, that is,
\begin{align}
	\label{eq:set_egf_excluded_cluster_sizes_hayman}
	S_{\neq\mathrm{k}}(x)
	= S(x) \cdot T_1(x),
	\quad \text{where} \quad
	T_1(x) = \exp\bigg\{-\sum_{1\le i\le \ell}c_{k_i}x^{k_i}\bigg\}.
\end{align}
Note that for any $s_n<k<s_n+o(s_n)$ and by plugging in $s_n$ we get analogous to~\eqref{eq:set_max_A_1_1_computation_hayman} (where $\beta_n=\eta_n$ for $\mathsf{F}^{(n)}=\mathsf{S}^{(n)}$) that $c_kz_n^k = h(k)k^{\al-1}\eul^{-\eta_n k} \le  h(s_n)s_n^{\al-1}\eul^{-\eta_n s_n}(1+o(1))\sim \eta_n=o(1)$. Hence 
\begin{align}
	\label{eq:T1(w)_sim_1_hayman}
	T_1(z_n)\sim 1.
\end{align} 
Writing $\mathsf{S}^{(n)} = (\mathsf{S}_1^{(n)},\dots, \mathsf{S}_n^{(n)})$, we conclude
\begin{align}
	\label{eq:set_cluster_as_sum_lower_bound_hayman}
	\pr{\mathsf{S}^{(n)} \in \Omega_{n,\mathrm{k}}} 
	\ge \pr{\forall 1\le i\le \ell: \mathsf{S}^{(n)}_{k_i} = 1}
	= \frac{[x^{n-\Sigma_{\mathrm{k}}}]S_{\neq \mathrm{k}}(x)}{[x^n]S(x)} \cdot \prod_{1\le i\le \ell}c_{k_i}.
\end{align}
We compute
\begin{align}
	\label{eq:set_cluster_lower_bound_1}
	\frac{[x^{n-\Sigma_{\mathrm{k}}}]S_{\neq \mathrm{k}}(x)}{[x^n]S(x)} 
	= \frac{[x^{n-\Sigma_{\mathrm{k}}}]S(x)}{[x^n]S(x)} 
	+ \sum_{1\le u\le n} \frac{[x^{n-\Sigma_{\mathrm{k}} - u}]S(x)}{[x^n]S(x)}
	[x^u]T_1(x).
\end{align}
Analogous to~\eqref{eq:set_cluster_as_sum_upper_bound_hayman} and~\eqref{eq:set_cluster_as_sum_upper_bound_2_hayman} we obtain due to~\eqref{eq:T1(w)_sim_1_hayman}
\begin{align}
	\label{eq:set_cluster_lower_bound_2}
	\sum_{1\le u\le n} \frac{[x^{n-\Sigma_{\mathrm{k}} - u}]S(x)}{[x^n]S(x)}
	[x^u]T_1(x)
	=\mathcal{O}\bigg( z_n^{\Sigma_{\mathrm{k}}} \sum_{u\ge 1}[x^u]T_1(x) z_n^u\bigg)
	=\mathcal{O}\big( z_n^{\Sigma_{\mathrm{k}}} T_1(z_n)\big)
	= o\big(z_n^{\Sigma_{\mathrm{k}}} \big).
\end{align}
Next we show that the first term on the right-hand side of~\eqref{eq:set_cluster_lower_bound_1} is asymtotically equal to $z_n^{\Sigma_{\mathrm{k}}}$. 
Since $\Sigma_{\mathrm{k}} = \mathcal{O}(\eta_n^{-1}\ln C(z_n)) = o(n)$ for $s_n<k_1<\cdots <k_\ell<s_n+o(s_n)$ due to~\eqref{eq:set_max_xi_n_hayman} and~\eqref{eq:s(t,beta)_asymptotic_hayman} we obtain by Lemma~\ref{lem:hayman} as $n\to\infty$
\[
	\frac{[x^{n-\Sigma_{\mathrm{k}}}]S(x)}{[x^n]S(x)} 
	= z_n^{\Sigma_{\mathrm{k}}} 
	\bigg(\exp\bigg\{-\frac{\big(A_1(z_n)-(n-\Sigma_{\mathrm{k}})\big)^2}{2A_2(z_n)}\bigg\} + o(1) \bigg).
\]
Since $A_1(z_n)=n$ we just need to show that $\Sigma_{\mathrm{k}}^2/b(z_n)=o(1)$. We compute, noting that $A_2(z_n) =\Theta(h(\eta_n^{-1})\eta_n^{-(\al+2)})$ according to Lemma~\ref{lem:asymptotics_expansive_sum}, 
\begin{align}
	\label{eq:exp_hayman_o(1)_in_largest_cluster_proof}
	\frac{\Sigma_{\mathrm{k}}^2}{A_2(z_n)}
	= \bigO{\eta_n^{\al}\cdot \ln^2C(z_n)\cdot h(\eta_n^{-1})} 
	= o(1).
\end{align}
Concluding, we obtain by plugging in~\eqref{eq:set_cluster_lower_bound_2} into~\eqref{eq:set_cluster_lower_bound_1} that
\[
	\frac{[x^{n-\Sigma_{\mathrm{k}}}]S_{\neq \mathrm{k}}(x)}{[x^n]S(x)} 
	=(1+o(1)) \cdot z_n^{\Sigma_{\mathrm{k}}},\quad \text{for }s_n<k_1<\cdots<k_\ell<s_n+o(s_n) \text{ and as }n\to\infty,
\]
which in turn brings with~\eqref{eq:set_cluster_as_sum_lower_bound_hayman} that~\eqref{eq:set_cluster_as_sum_lower_bound_final_hayman} is true. This concludes the proof of~\eqref{eq:F_n_k_i_as_F_and_c_i_sim_hayman}. The set case is completed and we move on to the multiset case.

\paragraph{The multiset case.} 
This case is proven almost analogously to the set case, which is why we will be sparing with details.
Like in~\eqref{eq:set_cluster_as_sum_hayman} we get due to~\eqref{eq:distribution_cluster_multiset}
\[
	\pr{\mathsf{G}^{(n)} \in \Omega_{n,\mathrm{k}}}
	=\frac{1}{[x^n]G(x)}\cdot \prod_{1\le i\le \ell}c_{k_i}
	\cdot \sum_{\Omega_{n,\mathrm{k}}} \prod_{i\in[n]_{\neq \mathrm{k}}} \binom{c_i+N_i-1}{N_i} \prod_{1\le i\le\ell} \binom{c_{k_i}+N_{k_i}-1}{N_{k_i}}\frac1{c_{k_i}}.
\]
It is easy to check that $\binom{a+b-1}{b}/a \le \binom{a+b-2}{b-1}$ for $a,b\in\Nat$.
In addition $N_i=0$ for any $i>n-\Sigma_{\mathrm{k}}$ if $(N_1,\dots,N_n)\in\Omega_{n,\mathrm{k}}$. Thus, since $c_{k_i}, N_{k_i}\in\Nat$ for $1\le i\le \ell$ (otherwise the claims~\eqref{eq:F_n_k_i_as_F_and_c_i_sim_hayman}--\eqref{eq:F_n_k_i_as_F_and_c_i_le_hayman} are trivially true) we obtain that
\[
	\pr{\mathsf{G}^{(n)} \in \Omega_{n,\mathrm{k}}}
	\le \frac{[x^{n-\Sigma_{\mathrm{k}}}]G(x)}{[x^n]G(x)} \cdot \prod_{1\le i\le \ell}c_{k_i}.
\]
Let $w_n$ be given as in~\eqref{eq:def_w_beta_hayman}. Replacing $S$ by $G$ and $z_n$ by $w_n$ we obtain completely analogous to~\eqref{eq:set_cluster_as_sum_upper_bound_final_hayman} that for any $\eps>0$ and sufficiently large $n$
\[
	\pr{\mathsf{G}^{(n)} \in \Omega_{n,\mathrm{k}}}
	\le (1+\eps) \prod_{1\le i\le \ell}c_{k_i}w_n^{k_i}, 
	\quad \text{uniformly in }\mathrm{k}\in\Nat^\ell\text{ with }
	0< \Sigma_{\mathrm{k}} \le n,
\]
proving~\eqref{eq:F_n_k_i_as_F_and_c_i_le_hayman}. To finish the proof in the multiset case it suffices to show that
\begin{align}
	\label{eq:multiset_cluster_as_sum_lower_bound_final_hayman}
	\pr{\mathsf{G}^{(n)} \in \Omega_{n,\mathrm{k}}}
	\ge (1+o(1)) \cdot \prod_{1\le i\le \ell}c_{k_i}w_n^{k_i}
	\quad \text{for }s_n <k_1<\cdots k_\ell < s_n + o(s_n)\text{ and as }n\to\infty.
\end{align}
Let $G_{\neq \mathrm{k}}(x)$ be the generating series of elements such that there are no clusters of sizes $k_1,\dots,k_\ell$, that is,
\[
	G_{\neq \mathrm{k}}(x)
	= G(x) \cdot T_2(x), \quad\text{where }T_2(x) = \exp\bigg\{-\sum_{j\ge 1}\sum_{1\le i\le \ell}c_{k_i}x^{jk_i}\bigg\}.
\]
For any $s_n<k<s_n+o(s_n)$ and by plugging in $s_n$ we get analogous to~\eqref{eq:set_max_A_1_1_computation_hayman} that $c_kw_n^k = h(k)k^{\al-1}\eul^{-\beta_n k} \le  h(s_n)s_n^{\al-1}\eul^{-\beta_n s_n}(1+o(1))\sim \beta_n=o(1)$. Hence 
\begin{align}
	\label{eq:T2(w)_sim_1_hayman}
	T_2(w_n)\sim 1.
\end{align} 
Consequently, we are at the exact same starting point as in~\eqref{eq:set_egf_excluded_cluster_sizes_hayman} and~\eqref{eq:T1(w)_sim_1_hayman}. Analogous to~\eqref{eq:set_cluster_as_sum_lower_bound_hayman}--\eqref{eq:set_cluster_lower_bound_2} we thus obtain as $n\to\infty$ and for $s_n<k_1<\cdots<k_\ell<s_n+o(s_n)$ as $n\to\infty$
\begin{align}
	\label{eq:multiset_prob_as_prod_hayman}
	\pr{\mathsf{G}^{(n)}\in\Omega_{n,\mathrm{k}}} 
	\ge \prod_{1\le i\le \ell}c_{k_i} \cdot 
	\bigg( \frac{[x^{n-\Sigma_{\mathrm{k}}}]S(x)}{[x^n]S(x)} + o\big(z_n^{\Sigma_{\mathrm{k}}}\big)
	\bigg).
\end{align}
Recall the definition of $A_{s,t}(x)$ from~\eqref{eq:A_s_t_def_hayman_proof}. The functions $a,b$ from~\eqref{eq:a,b,c_hayman_def} for $f(x)=\ln G(x)$ are then given by $A_{1,1}(x),A_{2,2}(x)$, 
see also~\eqref{eq:tilde_a_b_ell_hayman_proof} with $\ell=0$. Since $\Sigma_{\mathrm{k}} = o(n)$ for $s_n<k_1<\cdots<k_\ell<s_n+o(s_n)$ we obtain by Lemma~\ref{lem:hayman}
\begin{align}
	\label{eq:multiset_hayman_applied}
	\frac{[x^{n-\Sigma_{\mathrm{k}}}]S(x)}{[x^n]S(x)}
	= w_n^{\Sigma_{\mathrm{k}}}
	\bigg( \exp\bigg\{ -\frac{\big(A_{1,1}(w_n)-(n-\Sigma_{\mathrm{k}})\big)^2}{2A_{2,2}(w_n)} \bigg\} +o(1) \bigg)
\end{align}
Start with $0<\rho<1$, i.e. $w_n=z_n$ in~\eqref{eq:def_w_beta_hayman}. Due to~\eqref{eq:A_s_t_0<rho<1_proof_hayman} we have in this setting $ A_{1,1}(z_n) = A_1(z_n) + \bigO{1}$ and $A_{2,2}(z_n)\sim A_2(z_n)$. Hence analogous to~\eqref{eq:exp_hayman_o(1)_in_largest_cluster_proof} we get 
\begin{align}
	\label{eq:set_max_exponent_o(1)_mset_rho<1_hayman}
	\frac{\big(\tilde{a}(z_n)-(n-\Sigma_{\mathrm{k}})\big)^2}{2\tilde{b}(z_n)}
	\sim \frac{\Sigma_{\mathrm{k}}^2}{A_2(z_n)}
	=o(1).
\end{align} 
Now consider $\rho=1$ in which case $w_n=q_n$, chosen such that $A_{1,1}(q_n) = n$. Since $n^{-\delta}\le h(n)\le  n^{\delta}$ for any $0<\delta<\al$ and $n$ sufficiently large due to~\eqref{eq:sv_estimate_hayman} we have that $(c_k)_{k\in\Nat}\in\F(\al-\delta,\al+\delta,\rho)$ for any $0<\delta<\al$. So, we compute with~\eqref{eq:A_s_t_bigO_proof_hayman} that $ A_{2,2}(q_n) = \Omega(\xi_n^{-(\al-\delta+2)})$ giving us with~\eqref{eq:set_max_xi_n_hayman} and~\eqref{eq:s(t,beta)_asymptotic_hayman} that
\begin{align}
	\label{eq:set_max_exponent_o(1)_mset_rho=1_hayman}
	\frac{\big(A_{1,1}(q_n)-(n-\Sigma_{\mathrm{k}})\big)^2}{A_{2,2}(q_n)}
	=\mathcal{O}\bigg( \frac{s_n^2}{A_{2,2}(q_n)} \bigg)
	= \bigO{\ln^2 C(z_n) \xi_n^{\al-\delta}}
	=o(1).
\end{align}
Plugging~\eqref{eq:multiset_hayman_applied}--\eqref{eq:set_max_exponent_o(1)_mset_rho=1_hayman} into~\eqref{eq:multiset_prob_as_prod_hayman} yields~\eqref{eq:multiset_cluster_as_sum_lower_bound_final_hayman} and we are done.
\end{proof}

\begin{proof}[Proof of Example~\ref{example:largest_comp}]
We get by~\cite[Lem.~4.2]{Erlihson2004} that there is a $A:\Real^+\to\Real^+$  such that
\[
	C(z_n) = \Gamma(\al)(\eta_n^{-\al}+A(\al)) + \bigO{\eta_n} \quad\text{and}\quad 
	z_nC'(z_n) = \Gamma(\al+1)(\eta_n^{-(\al+1)}+A(\al+1)) + \bigO{\eta_n}.
\]
This immediately gives us that $z_nC'(z_n)=n$ implies $\eta_n = \Gamma(\al+1)^{1/(\al+1)}n^{-1/(\al+1)} +o(n^{-1})$. Plugging this into $C(z_n)$ yields $C(z_n) =\Gamma(\al)\Gamma(\al+1)^{-\al/(\al+1)}n^{\al/(\al+1)} + \bigO{1}$. Hence, setting $f(n) = (n/\Gamma(\al+1))^{1/(\al+1)}$,
\[
	\ln X = \ln \big( \Gamma(\al)^{-1}C(z_n) (\ln C(z_n))^{\al-1}\big)
	= \al\ln f(n) + (\al-1)\ln\ln f(n) + (\al-1)\ln \al + o(1).
\] 
Next consider $\rho=1$ in the multiset case. We want to compute the solution $q_n=\eul^{-\xi_n}$ to $\sum_{j\ge 1}q_n^jC'(q_n^j)=n$. Clearly $\xi_n\to 0$ as $n\to\infty$. For $j\le \xi_n^{-1}$ we obtain with~\cite[Lem.~4.2]{Erlihson2004} that $q_n^jC'(q_n^j) = \Gamma(\al+1)((j\xi_n)^{-(\al+1)} + A(\al+1)) + \bigO{j\xi_n}$. Hence
\begin{align*}
    \sum_{1\le j\le \xi_n^{-1}}q_n^jC'(q_n^j) 
    &=\Gamma(\al+1)\zeta(\al+1)\xi_n^{-(\al+1)} + \bigO{ \xi_n^{-(\al+1)} \sum_{j>\xi_n^{-1}}j^{-(\al+1)}} + \bigO{\xi_n^{-1}} \\
    &= \Gamma(\al+1)\zeta(\al+1)\xi_n^{-(\al+1)} + \bigO{\xi_n^{-1}}.
\end{align*}
We further get that $C'(q_n^j)=\sum_{k\ge 1}k^\al \eul^{-\xi_nj(k-1)}\le \eul \sum_{k\ge 1}k^\al \eul^{-k}=:B<\infty$. Hence
\[
    \sum_{j>\xi_n^{-1}}q_n^jC'(q_n^j)
    \le B \sum_{j>\xi_n^{-1}} \eul^{-j\xi_n}
    = \bigO{\xi_n^{-1}}.
\]
All in all, we conclude
\begin{align}
    \label{eq:example_eq}
    \sum_{j\ge 1}q_n^jC'(q_n^j)
    = \Gamma(\al+1)\zeta(\al+1)\xi_n^{-(\al+1)}  + \bigO{\xi_n^{-1}}.
\end{align}
Define $\tilde{f}(n) := (n/(\Gamma(\al+1)\zeta(\al+1))^{1/(\al+1)}$. Setting~\eqref{eq:example_eq} equal to $n$ and applying~\cite[Lem.~4.2]{Erlihson2004} yields that
\[
    \xi_n
    = \tilde{f}(n)^{-1} + \bigO{n^{-1}}
    \quad\text{and}\quad
    C(q_n)
    =\Gamma(\al)\tilde{f}(n)^\al + \bigO{1}.
\]
This finally gives us 
\[
    \ln X
    = \ln (\Gamma(\al)^{-1}C(z_n)(\ln C(z_n))^{\al-1})
    = \al \ln \tilde{f}(n) + (\al-1)\ln\ln\tilde{f}(n) + (\al-1)\ln\al + o(1).
\]
\end{proof}

\subsubsection*{Proof of Corollary~\ref{coro:smallest_comp_hayman}}
\label{sec:coro:smallest_comp_hayman}
\begin{proof}[Proof of Corollary~\ref{coro:smallest_comp_hayman}]
 Let $\mathsf{F}^{(n)}$ be either $\mathsf{S}^{(n)}$ or $\mathsf{G}^{(n)}$. Set $F(x)=S(x)$ if $\mathsf{F}^{(n)}=\mathsf{S}^{(n)}$ and $F(x)=G(x)$ if $\mathsf{F}^{(n)}=\mathsf{G}^{(n)}$. Further we define the generating series for all elements such that the smallest object is of size greater than $s$  by
\begin{align*}
	F_{>s}(x)
	:= \begin{dcases*}
	\exp\bigg\{\sum_{k>s}c_kx^k\bigg\}, &$\mathsf{F}^{(n)}=\mathsf{S}^{(n)}$ \\
	\exp\bigg\{\sum_{j\ge 1}\sum_{k>s}c_kx^{jk}/j\bigg\}, &$\mathsf{F}^{(n)}=\mathsf{G}^{(n)}$
	\end{dcases*}.
\end{align*}
Then
\[
	\pr{\mathcal{M}(\mathsf{F}^{(n)})>s}
	=\frac{[x^n]F_{>s}(x)}{[x^n]F(x)}.
\]
Since $(c_k)_{k\in\Nat}\in \F(2\al/3+\eps,\al,\rho)$ we also have that $(c_k)_{k>s}\in \F(2\al/3+\eps,\al,\rho)$ for fixed $s\in\Nat$. Then Lemma~\ref{lem:h_admissible_gen_series} reveals that both $F_{>s}$ and $F$ are $H$-admissible. Letting $a,b$ and $a_{>s},b_{>s}$ be the functions~\eqref{eq:a,b,c_hayman_def} we immediately see that $a(x)-a_{>s}(x)$ and $b(x)-b_{>s}(x)$ are bounded uniformly in $x<\rho$. Then Lemma~\ref{lem:hayman} gives for any $w_n\to\rho$ as $n\to\infty$
\begin{align*}
	\pr{\mathcal{M}(\mathsf{F}^{(n)})>s}
	=\frac{[x^n]F_{>s}(x)}{[x^n]F(x)}
	\sim \frac{F_{>s}(w_n)}{F(w_n)} \bigg( \e{-\frac{(a_{>s}(w_n)-n)^2}{2b_{>s}(w_n)}}
	\e{\frac{(a(w_n)-n)^2}{2b(w_n)}} + o(1)\bigg).
\end{align*}
Choosing $w_n$ as in Theorem~\ref{thm:coeff_eC(x)_C(x)^k} for the different cases depending on $S,G$ and $\rho$ as well as noting again that $a(w_n)-a_{>s}(w_n)=\bigO{1}$ we get that the exponents of the exponential functions in the previous display are $o(1)$ and so (as $s$ is fixed and $w_n\to\rho$)
\begin{align*}
	\pr{\mathcal{M}(\mathsf{F}_n)>s}
	\sim \frac{F_{>s}(w_n)}{F(w_n)}
	\sim \begin{dcases*}
	\exp\bigg\{-\sum_{1\le k\le s}c_k\rho^k\bigg\},&$\mathsf{F}^{(n)}=\mathsf{S}^{(n)}$ \\
	\exp\bigg\{-\sum_{j\ge 1}\sum_{1\le k\le s}c_k\rho^{jk}/j\bigg\},  &$\mathsf{F}^{(n)}=\mathsf{G}^{(n)}$
	\end{dcases*}.
\end{align*}
\end{proof}

\subsubsection{Proofs for the Cluster Distribution}

\subsubsection*{Proof of Corollary~\ref{coro:mean_sec_moment_hayman}} 
\label{sec:coro:mean_sec_moment_hayman}
\begin{proof}[Proof of Corollary~\ref{coro:mean_sec_moment_hayman} (set case).]
Due to~\eqref{eq:number_comp_falling_factorial_hayman} we have that for any $\ell\in\Nat$
\[
	\ex{(\kappa(\mathsf{S}^{(n)}))_\ell}=\frac{[x^n]\e{C(x)}C(x)^\ell}{[x^n]\e{C(x)}}.
\]
An application of Theorem~\ref{thm:coeff_eC(x)_C(x)^k} delivers 
\begin{align}
	\label{eq:falling_factorial_kappa_S_hayman}
    \ex{(\kappa(\mathsf{S}^{(n)}))_\ell}
    \sim C(z_n)^\ell.
\end{align}
In particular $\ex{\kappa(\mathsf{S}^{(n)})} \sim C(z_n)$, which is the starting point for our induction. Assume that $\ex{\kappa(\mathsf{S}^{(n)})^\ell} \sim C(z_n)^\ell$ for $\ell\in\Nat$.
There are constants $(d_1,\dots,d_{\ell})=(d_1(\ell),\dots,d_\ell(\ell)\in\Real^\ell$ such that
\begin{align*}
    \ex{(\kappa(\mathsf{S}^{(n)})_{\ell+1}} 
    &= \ex{\kappa(\mathsf{S}^{(n)})(\kappa(\mathsf{S}^{(n)})-1)\cdots (\kappa(\mathsf{S}^{(n)})-\ell+2)}  \\
    &= \ex{(\kappa(\mathsf{S}^{(n)})^{\ell+1}} + \sum_{1\le i\le \ell}d_i\ex{(\kappa(\mathsf{S}^{(n)})^i}
    \sim \ex{(\kappa(\mathsf{S}^{(n)})^{\ell+1}} + \sum_{1\le i\le \ell}d_iC(z_n)^i,
\end{align*}
where we used the induction hypothesis in the last asymptotic identity of the previous display.
Since~\eqref{eq:falling_factorial_kappa_S_hayman} reveals that $\ex{(\kappa(\mathsf{S}^{(n)})_{\ell+1}} \sim C(z_n)^{\ell+1}$ and $C(z_n)^{\ell+1} = \omega(C(z_n)^i)$ for all $1\le i\le \ell$ the claim follows. 
\end{proof}
\begin{proof}[Proof of Corollary~\ref{coro:mean_sec_moment_hayman} (multiset case).]
Next we want to compute $\ex{\kappa(\mathsf{G}^{(n)})^\ell}$ for $0<\rho<1$. Setting $(x)_0=1$ for any $x\in\Real$, define $B_k(x,y) := \sum_{j\ge 1}(j-1)_{k-1} C(x^j) y^{j-k}$. Then $d/dy G(x,y) = G(x,y)B_1(x,y)$ and $d/dy B_k(x,y) = B_{k+1}(x,y)$ for $k\in\Nat$. By a simple induction there exist real-valued constants $(d_{k_1,\dots,k_{\ell-1}})_{k_1,\dots,k_{\ell-1}\in\Nat_0}=(d_{k_1,\dots,k_{\ell-1}}(\ell))_{k_1,\dots,k_{\ell-1}\in\Nat_0}$ such that
\begin{align}
	\label{eq:derivative_G(x,y)_in_terms_B(x,y)_hayman}
	\frac{d^\ell}{dy^\ell}G(x,y)
	= G(x,y)\bigg( B_1(x,y)^\ell + \sum_{0\le k_1\le \cdots\le k_{\ell-1}\atop k_1+\cdots+k_{\ell-1} = \ell} d_{k_1,\dots,k_{\ell-1}} \prod_{1\le i\le \ell-1} B_{k_i}(x,y)\bigg).
\end{align}
Recall the definition of $A_{s,t}$ from~\eqref{eq:A_s_t_def_hayman_proof}. Clearly, for any $k\in\Nat$, we can rewrite $B_k(x,1)=A_{0,k}(x) + \sum_{1\le i \le k-1}b_iA_{0,i}(x)$ for some real-valued constants $(b_1,\dots, b_{\ell-1})=(b_1(\ell),\dots,b_{\ell-1}(\ell))$. Hence together with~\eqref{eq:derivative_G(x,y)_in_terms_B(x,y)_hayman} there are real-valued constants $(d'_{k_1,\dots,k_{\ell-1}})_{k_1,\dots,k_{\ell-1}\in\Nat_0}=(d'_{k_1,\dots,k_{\ell-1}}(\ell))_{k_1,\dots,k_{\ell-1}\in\Nat_0}$ such that
\begin{align}
	\label{eq:derivative_G(x,1)_in_terms_B(x,1)_hayman}
	\frac{d^\ell}{dy^\ell}G(x,y)\bigg|_{y=1}
	= G(x,y)\bigg( A_{0,1}(x)^\ell + \sum_{0\le k_1\le \cdots\le k_{\ell-1}\atop k_1+\cdots+k_{\ell-1} \le \ell} d'_{k_1,\dots,k_{\ell-1}} \prod_{1\le i\le \ell-1} A_{0,k_i}(x)\bigg).
\end{align}
Now~\eqref{eq:number_comp_falling_factorial_hayman} and Theorem~\ref{thm:coeff_eC(x)_C(x)^k} give us
\[
	\ex{(\kappa(\mathsf{G}^{(n)}))_\ell}
	=\frac{[x^n]d^\ell/(dy^\ell) G(x,y)}{[x^n]G(x,y)}
	\sim  A_{0,1}(z_n)^\ell + \sum_{0\le k_1\le \cdots\le k_{\ell-1}\atop k_1+\cdots+k_{\ell-1} = \ell} d'_{k_1,\dots,k_{\ell-1}} \prod_{1\le i\le \ell-1} A_{0,k_i}(z_n).
\]
Due to~\eqref{eq:A_s_t_0<rho<1_proof_hayman} we get that $A_{0,k}(z_n) \sim A_{0}(z_n)=C(z_n)$. Since for any $0\le k_1\le\cdots \le k_{\ell-1}$ we always have that $\prod_{1\le i\le \ell-1}A_{0,k_i}(z_n)\sim C(z_n)^{\ell'}$ for $\ell'<\ell$ and $C(z_n)^\ell = \omega(C(z_n)^{\ell'})$ we finally obtain that
\[
	\ex{(\kappa(\mathsf{G}^{(n)}))_\ell}
	\sim C(z_n)^\ell.
\]
The claim $\ex{\kappa(\mathsf{G}^{(n)})^\ell}\sim C(z_n)^\ell$ follows analogously to the induction after~\eqref{eq:falling_factorial_kappa_S_hayman}.

Let us now consider $\rho=1$. Here we would also get~\eqref{eq:derivative_G(x,1)_in_terms_B(x,1)_hayman} but we cannot simplify $A_{0,k}(z_n)\sim C(z_n)$ and thereby let all the terms but $A_{0,1}(z_n)^\ell$ asymptotically vanish; in fact, all the terms could (depending on $\al>0$) play a role. This is why we are content with only computing $\ex{\kappa(\mathsf{G}^{(n)})^\ell}$ for $\ell=1,2$ in this case. With~\eqref{eq:number_comp_falling_factorial_hayman} and Theorem~\ref{thm:coeff_eC(x)_C(x)^k} we have that
\[
	\ex{\kappa(\mathsf{G}^{(n)})} = \frac{[x^n]G(x)\sum_{j\ge 1}C(x^j)}{[x^n]G(x)}
	\sim \sum_{j\ge 1}C(q_n^j).
\]
We have that $\ex{(\kappa(\mathsf{G}^{(n)}))_2} = \ex{\kappa(\mathsf{G}^{(n)})^2} - \ex{\kappa(\mathsf{G}^{(n)})}$ such that due to~\eqref{eq:number_comp_falling_factorial_hayman} and Theorem~\ref{thm:coeff_eC(x)_C(x)^k}
\[
	\ex{(\kappa(\mathsf{G}^{(n)}))_2}
	= \frac{[x^n]G(x)\big(\sum_{j\ge1 }C(x^j)\big)^2}{[x^n]G(x)} 
	+ \frac{[x^n]G(x)\sum_{j\ge1 }jC(x^j)}{[x^n]G(x)}
	\sim \bigg(\sum_{j\ge1 }C(q_n^j)\bigg)^2 + \sum_{j\ge1 }jC(q_n^j).
\]
Since obviously $(\sum_{j\ge1 }C(q_n^j))^2 = \omega(\sum_{j\ge1 }C(q_n^j))$ the claim follows.
\end{proof}

\subsubsection*{Proof of Theorems~\ref{thm:LLT_number_comp_hayman} and~\ref{thm:bivariate_counting_set_hayman}} 
\label{sec:thm:bivariate_counting_set_hayman}
The proofs of the part concerning $\mathsf{S}^{(n)}$ in Theorem~\ref{thm:LLT_number_comp_hayman} and of Theorem~\ref{thm:bivariate_counting_set_hayman} are a straightforward consequence of the following local limit theorem, which we state for completeness.
\begin{lemma}[{\cite[Lem~3.7]{Panagiotou2022}}]
\label{lem:enumeration_paper_LLT_lemma_hayman}
Let $(c_k)_{k\in\Nat}$ be expansive. For $\chi>0$ set $q=\rho\eul^{-\chi}$. Let $C_1(\chi), C_2(\chi),\dots$ be iid integer-valued non-negative random variables with probability generating function $C(qx)/C(q)$. Define $S_p(\chi) := \sum_{1\le i\le p}C_i(\chi)$ for $p\in\Nat$ as well as
\[
    \nu_p(\chi) 
    := \ex{S_p(\chi)} 
    = p \frac{qC'(q)}{C(q)} \quad \text{and}\quad 
    \sigma_p(\chi)^2
    := \Var{S_p(\chi)}
	= p \left(\frac{{q}^2C''({q}) + qC'({q})}{C({q})} 
	-\left(\frac{{q}C'({q})}{C({q})}\right)^2\right).
\]
Then for $p=p(\chi)$ such that $p\to\infty$ as $\chi\to 0$ we obtain for $t=o(p^{1/6})$
\[
    \pr{ S_p(\chi) = \nu_p(\chi) + t\sigma_p(\chi) }
    \sim \eul^{-t^2/2} \cdot \frac{1}{\sqrt{2\pi}\sigma_p(\chi)}, \qquad\chi\to0.
\]
\end{lemma}
Since Theorem~\ref{thm:bivariate_counting_set_hayman} will be needed for the proof of the set case in Theorem~\ref{thm:LLT_number_comp_hayman} we start in anti-chronological order.
\begin{proof}[Proof of Theorem~\ref{thm:bivariate_counting_set_hayman}]
Write $r_n=\rho\eul^{-\varphi_n}$. Let $S_p=S_p(\varphi_n)$, $\nu_p=\nu_p(\varphi_n)$ and $\sigma_p^2=\sigma_p(\varphi_n)^2$ be all defined as in Lemma~\ref{lem:enumeration_paper_LLT_lemma_hayman} for $p\in\Nat$. Recalling $r_nC'(r_n)/C(r_n)=n/N$ we obtain that $\nu_{N}= n$. Summarizing,
\[
    [x^ny^N]S(x,y)
    =\frac{1}{N!}[x^n]C(x)^N 
    = \frac{r_n^{-n}C(r_n)^{N}}{N!}\cdot \pr{ S_{N} = n}
    = \frac{r_n^{-n}C(r_n)^{N}}{N!}\cdot \pr{ S_{N} = \nu_{N}}. 
\]
Since $r_nC'(r_n)/C(r_n)=n/N\to\infty$ we necessarily have that $\varphi_n\to0$ as $n\to\infty$ so that we are allowed to apply Lemma~\ref{lem:enumeration_paper_LLT_lemma_hayman} and obtain
\[ 
     \pr{ S_{N} = \nu_{N}} 
     \sim \frac{1}{\sqrt{2\pi }\sigma_N}.
\]
Making use of Lemma~\ref{lem:asymptotics_expansive_sum} we finish the proof by computing 
\[
	\sigma_N^2
	= \frac{N}{C(r_n)} \big(r_nC''(r_n)+r_nC'(r_n) - (r_nC'(r_n))^2/C(r_n)\big)
	\sim N\frac{r_n^2C''(r_n)}{(\al+1)C(r_n)}.
\]
\end{proof}

\begin{proof}[Proof of Theorem~\ref{thm:LLT_number_comp_hayman} (set case).]
Here we are going to prove the local limit theorem for $\kappa(\mathsf{S}^{(n)})$. Set $N'=N'(n,t):=C(z_n) + L$ where $L = \floor{C(z_n) + t\sqrt{C(z_n)/(\al+1)}}-C(z_n) = t\sqrt{C(z_n)/(\al+1)} + \bigO{1}$. We want to determine $\pr{\kappa(\mathsf{S}^{(n)}) = N'} = [x^ny^{N'}]S(x,y)/[x^n]S(x)$. Let $z_n= \rho\eul^{-\eta_n}$ be such that $z_nC'(z_n)=n$ (implying that $\eta_n\to0$ as $n\to\infty$). Let $S_p=S_p(\eta_n),\nu_p=\nu_p(\eta_n)$ and $\sigma_p=\sigma_p(\eta_n)$ be as in Lemma~\ref{lem:enumeration_paper_LLT_lemma_hayman} for $p\in\Nat$. Then we obtain
\begin{align}
	\label{eq:coeff_S(x,y)_in_LLT_proof_hayman}
	[x^ny^{N'}]S(x,y) = \frac{z_n^{-n}C(z_n)^{N'}}{N'!}  \pr{S_{N'} = n}.
\end{align}
We have $\nu_{N'} = (C(z_n) + L) z_nC'(z_n)/C(z_n) = n + Lz_nC'(z_n)/C(z_n)$. Further as $N'\sim C(z_n)$ we get with Lemma~\ref{lem:asymptotics_expansive_sum} $Lz_nC'(z_n)/C(z_n)\sim t\cdot\sqrt{\al/(\al+1)}\cdot \sqrt{z_n^2C''(z_n)/(\al+1)} \sim t \cdot \sqrt{\al/(\al+1)}\cdot \sigma_{N'}$. Hence Lemma~\ref{lem:enumeration_paper_LLT_lemma_hayman} delivers for any $K>0$ and uniformly in $t\in[-K,K]$
\[
	\pr{S_{N'} = n}
	= \pr{S_{N'} = \nu_{N'} - (t\sqrt{\al/(\al+1)}+o(1))\sigma_{N'}}
	\sim \eul^{-t^2\al/(2(\al+1))} \frac{1}{\sqrt{2\pi z_n^2C''(z_n)/(\al+1)}}.
\]
We treat the remaining terms in~\eqref{eq:coeff_S(x,y)_in_LLT_proof_hayman} by Stirling's formula and using that $(1+a)^b=\e{b\ln(1+a)}=\e{b(a-a^2/2+a^3/3-\cdots)}$ for $b>0,0<a<1$ which gives us
\[
	\frac{C(z_n)^{N'}}{N'!}
	\sim \frac{\eul^{N'}}{\sqrt{2\pi C(z_n)}} \bigg(1+ \frac{L}{C(z_n)}\bigg)^{-N'}
	\sim \frac{\eul^{N'}}{\sqrt{2\pi C(z_n)}} \eul^{-L-t^2/(2(\al+1))}.
\]
Plugging everything back together yields
\[
	[x^ny^{N'}]S(x,y)
	\sim \eul^{-t^2/2} \frac{\e{C(z_n)}}{2\pi\sqrt{C(z_n)z_n^2C''(z_n)/(\al+1)}}\cdot z_n^{-n}.
\]
The claim follows by computing $[x^n]S(x) \sim \e{C(z_n)}/\sqrt{2\pi z_n^2C''(z_n)}\cdot z_n^{-n}$ due to Theorem~\ref{thm:coeff_eC(x)_C(x)^k} and dividing $[x^ny^{N'}]S(x,y)/[x^n]S(x)$.
\end{proof}
\begin{proof}[Proof of Theorem~\ref{thm:LLT_number_comp_hayman} (multiset case).]
Next we show the local limit theorem for $\kappa(\mathsf{G}^{(n)})$.
We write $N' =N'(n,t) := C(z_n) + L$ where $L=\floor{C(z_n) +t\sqrt{C(z_n)/(\al+1)}} -C(z_n) = t\sqrt{C(z_n)/(\al+1)} + \bigO{1}$.
 In what follows we want to determine the probability $\pr{\kappa(\mathsf{G}^{(n)})=N'} =[x^ny^{N'}]G(x,y)/[x^n]G(x)$. For that we need to repeat some notation from~\cite{Panagiotou2022}. Let $m$ be the first index such that $c_m>0$. For $n,N\in\Nat$ let $x,y$ be the solution to the system of equations 
\begin{align}
\label{eq:saddle_point_equation_bivariate_hayman}
	xyC'(x) + mc_m\frac{x^my}{1-x^my}=n, \quad yC(x) + c_m\frac{x^my}{1-x^my}=N
	\quad\text{and}\quad x^my<1.
\end{align}
Further let $u$ the solution to the system in the variable $v$
\begin{align}
\label{eq:def_slowly_g_hayman}
	u h(u)^{1/(\al+1)} = v^{1/(\al+1)}.
\end{align}
Then~\cite[Lem. 2.1]{Panagiotou2022} says that for $n,N,n-mN$ and $v$ sufficiently large there are unique solutions $x_{n,N},y_{n,N}$ and $u_v$ solving~\eqref{eq:saddle_point_equation_bivariate_hayman} and~\eqref{eq:def_slowly_g_hayman}, respectively. In particular, there is a slowly varying function $g:\Real^+\to\Real^+$ such that $u_v = v^{1/(\al+1)}/g(v)$. With this at hand, define
\[
	N_n^*
	=C_0\cdot g(n) \cdot n^{\al/(\al+1)},
	\quad\text{where } C_0
	:=\al^{-1}(\rho^{-m}\Gamma(\al+1))^{1/(\al+1)}.
\]
Depending on $\limsup N/N_n^*<1$ or $\liminf N/N_n^*>1$ there is a phase transition at which $[x^ny^N]G(x,y)$ switches its asymptotic behaviour, see~\cite[Thm.~1(I),~1(II)]{Panagiotou2022}. First we are going to show that 
\begin{align}
	\label{eq:C(z_n)_with_slowly_varying_g_hayman}
	C(z_n)
	\sim \al^{-1}\Gamma(\al+1)^{1/(\al+1)} \cdot g(n) \cdot n^{\al/(\al+1)}.
\end{align}
This implies that $C(z_n)/N_n^* \sim \rho^{m/(\al+1)}<1$ so that we are able to use all the results for Case~$(I)$ in~\cite{Panagiotou2022}, in particular Theorem~$1(I)$ for the determination of $[x^ny^{N'}]G(x,y)$. Since $z_nC'(z_n) = n$ we necessarily have that $z_n=\rho\eul^{-\eta_n}$ such that $\eta_n\to 0$ as $n\to\infty$ due to Lemma~\ref{lem:asymptotics_expansive_sum}, which also gives us $z_nC'(z_n) \sim \Gamma(\al+1)h(\eta_n^{-1})\eta_n^{-(\al+1)}$. Hence
\[
	\eta_n^{-1} h(\eta_n^{-1})^{1/(\al+1)}
	= (n/\Gamma(\al+1))^{1/(\al+1)}(1+o(1)).
\]
Since $n/\Gamma(\al+1)\to\infty$ it directly follows from~\eqref{eq:def_slowly_g_hayman} and the subsequent text that $\eta_n^{-1} \sim (n/\Gamma(\al+1))^{1/(\al+1)}/g(n)$. We also deduce that $g(n)^{\al+1}\sim h(\eta_n^{-1})$ by comparing the two representations of $\eta_n^{-1}$. Hence Lemma~\ref{lem:asymptotics_expansive_sum} yields
\[
	C(z_n)
	\sim \Gamma(\al)h(\eta_n^{-1})\eta_n^{-\al} 
	\sim \al^{-1}\Gamma(\al+1)^{1/(\al+1)}\cdot g(n) \cdot n^{\al/(\al+1)}.
\]
Consequently $\limsup N'/N_n^* < 1$ and $N'/N_n^*=\Theta(1)$. Let $(x_n,y_n)=(x_{n,N'},y_{n,N'})$ be the solution to~\eqref{eq:saddle_point_equation_bivariate_hayman}. Then~\cite[Thm.~1(I)]{Panagiotou2022} together with Theorem~\ref{thm:coeff_eC(x)_C(x)^k} reveal that
\begin{align}
	\label{eq:g_nN'_div_by_g_n_hayman}
	\frac{g_{n,N'}}{g_n}
	\sim \frac{\e{\sum_{j\ge 2}C(\rho^j)y_n^j/j}}{\e{\sum_{j\ge2}C(\rho^j)/j}}
	\cdot \frac{\e{y_nC(x_n) - C(z_n)}}{\sqrt{2\pi N'y_nx_n^2C''(x_n)/((\al+1)z_n^2C''(z_n)) }} \cdot \left(\frac{z_n}{x_n}\right)^n \cdot y_n^{-N'}.
\end{align}
In the remaining proof we show that $z_n/x_n$ and $y_n$ are so close to $1$ that the right-hand side of~\eqref{eq:g_nN'_div_by_g_n_hayman} is asymptotically $(2\pi N' /(\al+1))^{-1/2} \eul^{-t^2/2}$. For that we first repeat some important properties of $(x_n,y_n)$ from~\cite[Lem.~4.1]{Panagiotou2022}, that is,
\begin{align}
	\label{eq:S=O(1)_hayman}
	x_n \sim \rho,\quad
	\limsup_{n\to\infty} y_n < \rho^{-m}
	\quad\text{and}\quad 
	S_n:=\frac{x_n^my_n}{1-x_n^my_n} = \Theta(1).
\end{align}
Parametrize $x_n = z_n \eul^{\delta_n}$ for an appropriate $\delta_n$. We first show that 
\begin{align}
	\label{eq:delta_n_o_eta_n_hayman}
	\delta_n=o(\eta_n).
\end{align}
By~\eqref{eq:saddle_point_equation_bivariate_hayman} and~\eqref{eq:S=O(1)_hayman} we have that $y_nC(x_n) \sim C(z_n) + L \sim C(z_n)$ and $x_ny_nC'(x_n)\sim n = z_nC'(z_n)$. Plugging in $x_n=z_n\eul^{\delta_n} = \rho\eul^{-(\eta_n-\delta_n)}$ we obtain by Lemma~\ref{lem:asymptotics_expansive_sum} 
\[
	C(z_n) \sim 
	y_nC(x_n) 
	\sim y_n \Gamma(\al)h(\eta_n-\delta_n)(\eta_n-\delta_n)^{-\al}
	\sim y_n C(z_n) \frac{h(\eta_n-\delta_n)}{h(\eta_n)} \frac{(\eta_n-\delta_n)^{-\al}}{\eta_n^{-\al}}
\]
implying that
\begin{align}
	\label{eq:N'_sim_1_hayman}
	y_n \frac{h(\eta_n-\delta_n)}{h(\eta_n)} \frac{(\eta_n-\delta_n)^{-\al}}{\eta_n^{-\al}} \sim 1.
\end{align}
Analogously
\[
	n \sim x_ny_nC'(x_n)
	\sim y_n z_nC'(z_n) \frac{h(\eta_n-\delta_n)}{h(\eta_n)} \frac{(\eta_n-\delta_n)^{-(\al+1)}}{\eta_n^{-(\al+1)}}
\]
implying that
\begin{align}
	\label{eq:n_sim_1_hayman}
	y_n \frac{h(\eta_n-\delta_n)}{h(\eta_n)} \frac{(\eta_n-\delta_n)^{-(\al+1)}}{\eta_n^{-(\al+1)}} \sim 1.
\end{align}
Combining~\eqref{eq:N'_sim_1_hayman} and~\eqref{eq:n_sim_1_hayman} we obtain that $(\eta_n-\delta_n)/\eta_n \sim 1$ implying~\eqref{eq:delta_n_o_eta_n_hayman}. 

In what follows we use, without mentioning it every time, that $\delta_n=o(\eta_n)$ and Lemma~\ref{lem:asymptotics_expansive_sum} imply
\[
	z_n^kC^{(k)}(z_n)\delta_n 
	= \Theta(h(\eta_n^{-1})\eta_n^{-(\al+k)} \delta_n)
	=o(h(\eta_n^{-1})\eta_n^{-(\al+k-1)})
	=o(C^{(k-1)}(z_n)),
	\quad k\in\Nat.
\]
Next we expand $C(z_n\eul^{\delta})$ at $\delta=0$. Since $\delta_n=o(\eta_n)$ and $y_n=\bigO{1}$ we obtain that
\begin{align}
	\label{eq:y_nC(x_n)_expanded_more_orders_hayman}
	y_nC(x_n)
	&= y_nC(z_n) + y_n z_nC'(z_n) \delta_n +  y_nz_n^2C''(z_n)\delta_n^2/2 + o(C''(z_n)\delta_n^2) \\ \nonumber
	&=y_nC(z_n) + y_n z_nC'(z_n) \delta_n+	o(y_nC'(z_n)\delta_n).
\end{align}
Then the second identity in~\eqref{eq:saddle_point_equation_bivariate_hayman} gives
\[
	N'
	= y_nC(z_n) + y_nz_nC'(z_n)\delta_n + c_mS_n + o(C'(z_n)\delta_n) . 
\]
Recalling that $N'=C(z_n)+L$ and dividing both sides by $C(z_n)$ entails that $y_n\sim 1$ and
\begin{align}
	\label{eq:y_n_second_order_hayman}
	y_n  = 1 +\frac{L - c_m S}{C(z_n)} - \frac{z_nC'(z_n)}{C(z_n)}\delta_n + o\bigg(\frac{C'(z_n)}{C(z_n)}\delta_n\bigg).
\end{align}
We proceed similarly with the second identity in~\eqref{eq:saddle_point_equation_bivariate_hayman}. Expanding $z_n\eul^{\delta}y_nC'(z_n\eul^{\delta})$ around $\delta=0$ and using~\eqref{eq:delta_n_o_eta_n_hayman} yields
\begin{align}
	\nonumber
	x_ny_nC'(x_n)
	&= y_nz_nC'(z_n) + y_n(z_nC'(z_n) + z_n^2C''(z_n))\delta_n + \bigO{C'''(z_n)\delta_n^2} \\
	\label{eq:x_ny_nC'(x_n)_as_z_n_hayman}
	&=y_n n + y_n z_nC''(z_n) \delta_n + o(C''(z_n)\delta_n).
\end{align}
Note that Lemma~\ref{lem:asymptotics_expansive_sum} implies that $C''(z_n)=\Theta(C'(z_n)^2/C(z_n)) = \Theta(nC'(z_n)/C(z_n))$. Keeping this in mind, we plug in
\eqref{eq:x_ny_nC'(x_n)_as_z_n_hayman} and $y_n$ from~\eqref{eq:y_n_second_order_hayman} into the first equation of~\eqref{eq:saddle_point_equation_bivariate_hayman} to obtain
\begin{align*}
	n
	&= y_n n + y_nz_n^2C''(z_n)\delta_n + o(C''(z_n)\delta_n) + mc_mS_n \\
	&= n + n\frac{L-c_mS}{C(z_n)} +\delta_n\bigg( z_n^2C''(z_n)- n\frac{z_nC'(z_n)}{C(z_n)}\bigg) + o(C''(z_n)\delta_n) + mc_mS_n.
\end{align*}
Since $S_n=\Theta(1)$ due to~\eqref{eq:S=O(1)_hayman} this implies together with Lemma~\ref{lem:asymptotics_expansive_sum} that
\[
	\delta_n \sim -n\frac{L-c_mS}{C(z_n)} \bigg(z_n^2C''(z_n)-n\frac{z_nC'(z_n)}{C(z_n)}\bigg)^{-1} 
	\sim - \al\frac{L-c_mS}{ n}.
\]
This, in turn, implies with~\eqref{eq:y_n_second_order_hayman} that
\[
	y_n
	= 1 + (\al+1)\frac{L-c_mS_n}{C(z_n)} + o\bigg( \frac{L-c_mS_n}{C(z_n)}\bigg).
\]
It follows that there are $\al_n=o(L/n)$ and $\beta_n=o(L/C(z_n))$ such that
\begin{align}
	\label{eq:delta_y_n_hayman}
	\delta_n
	= - \al\frac{L-c_mS_n}{n} + \al_n \quad\text{and}\quad
	y_n
	= 1 + (\al+1)\frac{L-c_mS_n}{C(z_n)} + \beta_n.
\end{align}
Recall that $L\sim t\sqrt{C(z_n)/(\al+1)}$ and $S_n=\Theta(1)$ due to~\eqref{eq:S=O(1)_hayman}. Plugging this and the expressions for $\delta_n,y_n$ into~\eqref{eq:y_nC(x_n)_expanded_more_orders_hayman} as well as using Lemma~\ref{lem:asymptotics_expansive_sum} for $C(z_n),z_nC'(z_n)=n,z_n^2C''(z_n)$ yields
\begin{align*}
	y_nC(x_n)
	&=
	y_nC(z_n) + y_n z_nC'(z_n) \delta_n +  y_nz_n^2C''(z_n)\delta_n^2/2 + o(C''(z_n)\delta_n^2) \\
	&=C(z_n) +(\al+1)(L-c_mS_n) + \beta_nC(z_n) + 
	-\al(L-c_mS) + \al_n n \\ 
	&\quad-\al(\al+1)\frac{L^2}{C(z_n)} 
	+  z_n^2C''(z_n)\frac{\al^2}{2}\frac{L^2}{n^2}
	+o(1) \\
	& = C(z_n) + L - c_mS_n + \beta_n C(z_n) + \al_n n - \frac{\al}{2}t^2 + o(1).
\end{align*}
Plugging this into the second identity of~\eqref{eq:saddle_point_equation_bivariate_hayman} entails
\begin{align}
	\label{eq:beta_al_n_identity_hayman}
	N' = y_nC(x_n) + c_mS_n 
	\Rightarrow 
	\beta_n C(z_n) + \al_n n = \frac{\al}{2}t^2 + o(1).
\end{align}
With~\eqref{eq:delta_y_n_hayman} at hand we obtain that
\[
	\bigg(\frac{z_n}{x_n}\bigg)^n
	=\eul^{-\delta_n n}
	=\eul^{\al (L-c_m S_n) - \al_n n}
\]
and
\[
	y_n^{-N'}
	\sim \bigg(1 + (\al+1)\frac{L-c_mS_n}{C(z_n)}+\beta_n\bigg)^{-C(z_n)+t\sqrt{C(z_n)/(\al+1)}}
	\sim \eul^{-(\al+1)(L-c_mS_n) - t^2 - \beta_n C(z_n) + (\al+1)t^2/2}.
\]
Combining the previous two displays with~\eqref{eq:beta_al_n_identity_hayman} delivers
\[
	\bigg(\frac{z_n}{x_n}\bigg)^n
	y_n^{-N'} 
	= \eul^{\al (L-c_m S_n) - \al_n n-(\al+1)(L-c_mS_n) - t^2 - \beta_n C(z_n) + (\al+1)t^2/2}
	\sim \eul^{L-c_mS_n - t^2/2}.
\]
From $N'=y_nC(x_n) + c_mS_n$ we directly get $y_nC(x_n)-C(z_n) = L-c_mS_n$ so that
\[
	\e{y_nC(x_n)-C(z_n)}\cdot \bigg(\frac{z_n}{x_n}\bigg)^n\cdot y_n^{-N'}
	\sim \eul^{-t^2/2}.
\]
Since $\delta_n=o(\eta_n)$ as showed in~\eqref{eq:delta_n_o_eta_n_hayman} we obtain that $z_n^2C''(z_n)\sim x_n^2C''(x_n)$. Concluding, and plugging in $y_n\sim 1$, we obtain in~\eqref{eq:g_nN'_div_by_g_n_hayman}
\[
	\frac{g_{n,N'}}{g_n}
	\sim \frac{1}{\sqrt{2\pi N'/(\al+1)}} \eul^{-t^2/2}
	\sim \frac{1}{\sqrt{2\pi C(z_n)/(\al+1)}} \eul^{-t^2/2}
\]
as claimed.
\end{proof}

\bibliographystyle{abbrvnat}
\bibliography{references}

\end{document}